\documentclass[12pt]{amsart}
\usepackage{amsmath,amssymb,latexsym, amsthm, amscd, mathrsfs, stmaryrd, mathtools}
\usepackage[linktocpage=true]{hyperref}
\usepackage{cleveref}
\usepackage{dynkin-diagrams}
\usepackage{xcolor}
\usepackage{graphicx}
\usepackage[all]{xy}
\usepackage{enumitem}
\usepackage{chngcntr}
\usepackage{nicematrix}
\usepackage{multirow}
\usepackage{array}
\usepackage{booktabs}
\usepackage{tikz}
\usepackage{tikz-cd}
\usetikzlibrary{arrows.meta}
\usepackage{bbm}
\usepackage{longtable}
\usepackage{hyperref}
\usepackage[sorting=nyt,style=alphabetic,backend=bibtex,hyperref=true,doi=false,maxbibnames=9,maxcitenames=4,url=false,maxcitenames=50,maxbibnames=50]{biblatex}
\crefformat{section}{\S#2#1#3} 
\crefformat{subsection}{\S#2#1#3}
\crefformat{subsubsection}{\S#2#1#3}
\renewbibmacro{in:}{}
\newtheorem{thm}{Theorem} [section]

\theoremstyle{definition}

\newcommand{\bal}{{\mbox{\boldmath$\alpha$}}}
\newtheorem{Def}[thm]{Definition}
\newtheorem{Question}[thm]{Question}
\newtheorem{rem}[thm]{Remark}
\theoremstyle{plain}
\newtheorem{prop}[thm]{Proposition}

\newtheorem{lem}[thm]{Lemma}
\newtheorem{cor}[thm]{Corollary}

\numberwithin{equation}{section}

\newcommand{\Hom}{\mathrm{Hom}}

\newcommand{\KK}{\mathbb K}

\makeatletter
\newcommand{\extp}{\@ifnextchar^\@extp{\@extp^{\,}}}
\def\@extp^#1{\mathop{\bigwedge\nolimits^{\!#1}}}
\makeatother

\makeatletter
\newcommand*{\@rowstyle}{}
\newcommand*{\rowstyle}[1]{
  \gdef\@rowstyle{#1}%
  \@rowstyle\ignorespaces%
}
\newcolumntype{=}{
  >{\gdef\@rowstyle{}}%
}
\newcolumntype{+}{
  >{\@rowstyle}%
}
\makeatother

\let\bbordermatrix\bordermatrix
\patchcmd{\bbordermatrix}{8.75}{4.75}{}{}
\patchcmd{\bbordermatrix}{\left(}{\left[}{}{}
\patchcmd{\bbordermatrix}{\right)}{\right]}{}{}

\DeclareMathOperator*{\im}{im}

\DeclareMathOperator*{\Rep}{Rep}

\DeclareMathOperator*{\Vecc}{Vec}

\DeclareMathOperator*{\sVec}{sVec}
\DeclareMathOperator{\Ver}{Ver}
\DeclareMathOperator{\ver}{Ver}

\newcommand{\un}{\mathbbm{1}}

\pgfdeclarearrow{
name = jibladze,
parameters = { \the\pgfarrowlength },
setup code = {},
drawing code = {
  \pgfsetdash{}{0pt} 
      \pgfsetroundjoin   
  \pgfsetroundcap    
  \pgfsetlinewidth{5\pgflinewidth}
  \pgfsetstrokecolor{white}
  \pgfpathmoveto{\pgfpoint{-.75\pgfarrowlength}{\pgfarrowlength}}
  \pgfpathlineto{\pgfpoint{0}{0}}
  \pgfpathlineto{\pgfpoint{-.75\pgfarrowlength}{-\pgfarrowlength}}
  \pgfusepathqstroke
  \pgfsetlinewidth{.2\pgflinewidth}
  \pgfsetstrokecolor{black}
  \pgfpathmoveto{\pgfpoint{-.75\pgfarrowlength}{\pgfarrowlength}}
  \pgfpathlineto{\pgfpoint{0}{0}}
  \pgfpathlineto{\pgfpoint{-.75\pgfarrowlength}{-\pgfarrowlength}}
  \pgfusepathqstroke
},
defaults = { length = 2pt }
}

\title[Classification of Non-Degenerate Forms in $\Ver_4^+$]{Classification of Non-Degenerate Symmetric Bilinear and Quadratic Forms in the Verlinde Category $\Ver_4^+$}

\author[I. Chen]{Iz Chen}
\address{Los Altos High School, Los Altos, CA 94022}
\email{ichen4419@gmail.com}

\author[A.S. Kannan]{Arun S. Kannan}
\address{Department of Mathematics, Massachusetts Institute of Technology, Cambridge, MA 02139}
\email{akannan@mit.edu}

\author[K. Pothapragada]{Krishna Pothapragada}
\address{Naperville North High School, Naperville, IL 60563}
\email{krishnapothapragada2024@gmail.com }

\begin{document}

\begin{abstract}



Although Deligne’s theorem classifies all symmetric tensor categories (STCs) with moderate growth over algebraically closed fields of characteristic zero, the classification does not extend to positive characteristic. At the forefront of the study of STCs is the search for an analog to Deligne's theorem in positive characteristic, and it has become increasingly apparent that the Verlinde categories are to play a significant role. Moreover, these categories are largely unstudied, but have already shown very interesting phenomena as both a generalization of and a departure from superalgebra and supergeometry. In this paper, we study $\Ver_4^+$, the simplest non-trivial Verlinde category in characteristic $2$. In particular, we classify all isomorphism classes of non-degenerate symmetric bilinear forms and non-degenerate quadratic forms and study the associated Witt semi-ring that arises from the addition and multiplication operations on bilinear forms.
\end{abstract}

\maketitle

\setcounter{tocdepth}{1}
\tableofcontents

\section{Introduction}
\subsection{The broader picture: the quest for Deligne's theorem in positive characteristic}
While the study of the representation theory of groups initially started by finding and classifying individual representations, the modern perspective is to consider the category of all representations in totality. The notion of a \textit{symmetric tensor category} (always assumed to be of \textit{moderate growth} \footnote{A symmetric tensor category has \textit{moderate growth} if the lengths of tensor powers of every object are bounded by an exponential function. Although we will assume all STCs are of moderate growth, the study of STCs of non-moderate growth has also attracted attention (see \cite{deligne1982tannakian, deligne2002categories, deligne2007tannakiennes, etingof_complex_rank_2, harman2022oligomorphic} for examples of such categories).} in this paper) arises by axiomatizing the fundamental properties of representation categories of groups (see \cite{etingof2016tensor, etingof2021lectures} for basic details). A symmetric tensor category (STC) can be thought of as a ``home" to do commutative algebra and algebraic geometry without the language of vectors and vector spaces. One implication is that given an STC $\mathcal{C}$, we can construct affine group schemes over $\mathcal{C}$, whose representation categories give us other STCs. These are all said to \textit{fiber} over $\mathcal{C}$. Because it is shown in \cite{coulembier2023incompressible} that every STC fibers over a so-called \textit{incompressible} STC, it remains to classify the incompressible STCs.
\par
The STCs defined over an algebraically closed field $\KK$ of characteristic $p = 0$ are well-understood thanks to Deligne's theorem (see \cite{deligne2002categories, deligne2007tannakiennes}). This theorem states that, up to parity action, all manifestations of such STCs are simply representation categories of supergroup schemes, i.e.\@ they fiber over the category $\sVec_\KK$ of supervector spaces. This means the category $\Vecc_\KK$ of vector spaces and $\sVec_\KK$ are the only incompressible STCs in characteristic zero, and therefore, characteristic zero affords only (super)algebra and (super)geometry.
\par
As is par for the course, the story is completely different in positive characteristic. The most basic counterexample when the characteristic $p$ is larger than $3$ is the \textit{Verlinde category} $\Ver_p$, which contains $\sVec_\KK$ as a subcategory (see \cite{georgiev1994fusion, gelfand1992examples, ostrik2020symmetric}). This STC arises as the \textit{semisimplification} of the representation category $\Rep \bal_p = \Rep \KK[t]/(t^p)$ of the first Frobenius kernel $\bal_p$ of the additive group scheme $\mathbb{G}_a$ (cf.\@  \cite{etingof2021semisimplification}). It can be thought of as the positive-characteristic analog to $\Rep SL_2(\mathbb{C})$ with some truncation involved when taking tensor products. For instance, when $p = 5$, there is an object $X \in \Ver_5$ (which can be thought of as the analog of the adjoint representation of $SL_2\mathbb{C}$) that satisfies $\un \oplus X = X \otimes X$, where $\un$ is the unit object in the category. If this category were to fiber over supervector spaces, then $X$ would need to have integral dimension; this is impossible because there is no integral solution to $1 + \dim X = (\dim X)^2$.
\par
With Deligne's theorem failing in positive characteristic, much work has been done in recent years to find a suitable analog. The category $\Ver_p$ has served as a reasonable starting point: first, Ostrik proved in \cite{ostrik2020symmetric} that every semisimple STC fibers over $\Ver_p$, and this was later strengthened in \cite{coulembier2022frobenius} to say that an STC fibers over $\Ver_p$ if and only if it is \textit{Frobenius exact}. Indeed, the Verlinde category $\Ver_p$ sits in a larger sequence 
\[\Ver_p \subseteq \Ver_{p^2} \subseteq \cdots \subseteq \Ver_{p^\infty}\]
of incompressible STCs called the Verlinde categories. These were first discovered for $p = 2$ in \cite{benson2019symmetric} and then generalized for all $p > 0$ in \cite{benson2023new}. Therein, it is conjectured that the correct replacement for $\sVec_\KK$ in Deligne's theorem is $\Ver_{p^\infty}$, which is to say that every STC fibers over $\Ver_{p^\infty}$. 

\subsection{Content of this paper}
Although they arise out of the search for Deligne's theorem in positive characteristic, the Verlinde categories seem to be interesting objects in their own right as they exhibit new phenomena all the while generalizing the classical theory. For instance, in \cite{venkatesh_glx_2022}, the finite-length representations of the group scheme $GL(X)$ for an object $X \in \Ver_p$ are classified. Therein, the corresponding generalization of a torus no longer has one-dimensional representations, yet its representation theory is still semisimple.
\par
However, for the most part, these Verlinde categories have barely been studied. In this paper, we consider the simplest example in characteristic $2$, which is $\Ver_4^+$, a subcategory of $\Ver_4 = \Ver_{2^2}$ that was first shown to not fiber over the category of vector spaces in \cite{venkatesh_hilbert_basis} (note that $\Ver_2$ is just the category of vector spaces). We usually cannot use the language of vector spaces to describe objects in STCs, but as a tensor category, $\Ver_4^+$ is just $\Rep \KK[t]/(t^2)$  (and is therefore not semisimple). The symmetric structure, however, is different and arises from equipping the Hopf algebra $\KK[t]/(t^2)$ with a triangular structure (see \cite[\S8.3]{etingof2016tensor}) with $R$-matrix given by 

\[ R \coloneqq 1 \otimes 1 + t \otimes t.\]
In this category, we classify all alternating bilinear and all symmetric bilinear forms, up to isomorphism. We also describe how different isomorphism classes of bilinear forms interact when we take sum and product (after suitably defining such notions). 

Here, we say a form $B : U \otimes U \rightarrow \KK$ on an object $U \in \Ver_4^+$ is \textit{alternating} (resp.\@  \textit{symmetric}) if it vanishes on the kernel (resp.\@ image) of the map $1_{U \otimes U} - c_{U,U}$, where $c_{U,U} : U \otimes U \rightarrow U \otimes U$ is the braiding in this category given by
\[c_{U,U}(u \otimes u') = u' \otimes u + (t.u')\otimes(t.u)\] 
for $u, u' \in U$. In semisimple STCs like $\Ver_p$, the classification reduces to the vector space setting. In $\ver_4^+$, the presence of the two-dimensional indecomposable representation $P$ of $\KK[t]/(t^2)$ makes the classification more challenging. 
\par
We find that there are ultimately six families of non-degenerate symmetric bilinear forms, two of which are indexed by a parameter. We also calculate the Witt semi-ring, which is the semi-ring of isomorphism classes of non-degenerate symmetric bilinear forms.
\par
In \S\ref{forms_defs}, we define and establish basic properties about bilinear and quadratic forms in symmetric tensor categories. In \S\ref{basic_properties_section}, we define the Verlinde category $\Ver_4^+$ and offer descriptions of forms specific to this category. In Section \S\ref{classification}, we first classify non-degenerate symmetric bilinear forms on the object $nP$, then use this to recover the complete classification for an arbitrary object in $\Ver_4^+$ as well as the classification of non-degenerate quadratic forms. Finally, we describe the structure of the Witt semi-ring in Section \S\ref{witt}.

\subsection{Acknowledgements}
This paper is the result of MIT PRIMES-USA, a program that provides high school students an opportunity to engage in research-level mathematics and in which the second author mentored the first and third authors. The authors would like to thank the MIT PRIMES-USA program and its coordinators Prof.\@ Pavel Etingof, Dr.\@ Slava Gerovitch, and Dr.\@ Tanya Khovanova for providing the opportunity for this research experience. We would also like to thank Serina Hu for useful discussions. The second author would also like to thank Pavel Etingof for useful discussions and feedback. This paper is based upon work supported by The National Science Foundation Graduate Research Fellowship Program under Grant No.\@ 1842490 awarded to the second author.

\section{Forms in Symmetric Tensor Categories}\label{forms_defs}
In this section, we define bilinear and quadratic forms on objects in arbitrary symmetric tensor categories. From now on, when we say ``form'', we will always refer to either a bilinear form that is symmetric, skew-symmetric, or alternating, or a quadratic form (notions which we will define shortly). See \cite{etingof2016tensor,etingof2021lectures} for more details on symmetric tensor categories. We try to follow the notation in these references as close as possible.

\subsection{Symmetric, skew-symmetric, alternating, and quadratic forms}
\subsubsection{Extending definitions to arbitrary symmetric tensor categories}
Let us first remember what happens when working over the category of vector spaces. Given a vector space $V$ over a field $\mathbb{F}$, we say a bilinear form $\beta: V \otimes V \rightarrow \mathbb{F}$ is

\begin{enumerate}
    \item \textit{symmetric} if $\beta(v\otimes w) = \beta(w\otimes v)$ for all $v, w \in V$;
    \item \textit{skew-symmetric} if $\beta(v\otimes w) = -\beta(w\otimes v)$ for all $v, w \in V$;
    \item and \textit{alternating} if $\beta(v\otimes v) = 0$ for all $v \in V$.
\end{enumerate}
A map $q: V \rightarrow \mathbb{F}$ is a \textit{quadratic form} if $q(\lambda v) = \lambda^2 q(v)$ for all $v \in V, \lambda \in \KK$ and if the map $B_q: V \otimes V \rightarrow \KK$ given by $B_q(v\otimes w) = q(v+w) - q(v) - q(w)$ is symmetric. In order to extend these definitions to arbitrary symmetric tensor categories, we need to rephrase them categorically.
\par
Let $\mathcal{C}$ be any symmetric tensor category over $\mathbb{F}$ with braiding $c$. Recall that for any object $V \in \mathcal{C}$ the symmetric group $S_2$ on two symbols acts on $V\otimes V$ by sending the non-identity element to the morphism $c_{V,V}$. This defines various functors on $\mathcal{C}$:

\begin{enumerate}
    \item the second symmetric power functor $S^2$, where $S^2(V) = (V \otimes V)_{S_2} = (V \otimes V)/\im (1_{V\otimes V} - c_{V,V} )$;
    \item the second exterior power functor $\wedge^2$, where $\wedge^2(V) = (V \otimes V)/\im(1_{V\otimes V} + c_{V,V})$;
    \item and the second divided power functor $\Gamma^2$, where $\Gamma^2(V) = \ker(1_{V\otimes V} - c_{V,V})$.
\end{enumerate}
Now, looking at the definitions of symmetric, skew-symmetric, and alternating forms above, we can define these notions categorically as follows:

\begin{Def}
    Let $V \in \mathcal{C}$ be an object in a symmetric tensor category $\mathcal{C}$ over $\mathbb{F}$. A map $\beta: V \otimes V \rightarrow \mathbb{F}$ is called a \textit{bilinear form} on $V$. We say $\beta$ is

    \begin{enumerate}
        \item \textit{symmetric} if $\beta = \beta \circ c_{V,V}$ or equivalently $\beta$ factors through a map $S^2(V) \rightarrow \un$;
        \item \textit{skew-symmetric} if $\beta = -\beta \circ c_{V,V}$ or equivalently $\beta$ factors through a map $\wedge^2(V) \rightarrow \un$;
        \item and \textit{alternating} if $\beta|_{\Gamma^2(V)} = 0$ or equivalently it factors through a map $(V \otimes V)/\Gamma^2(V) \rightarrow \un$.
    \end{enumerate}
\end{Def}
The restriction $\beta|_W$ of a bilinear form $\beta$ on $V$ to a subobject $W$ is given by restricting to the subobject $W \otimes W \subseteq V \otimes V$. It is obvious that these definitions coincide with the usual definitions when working over vector spaces. We also see that any alternating form is skew-symmetric as $(1 - c_{V,V})(1 + c_{V,V}) = 0$, so a form that vanishes on the kernel of $1 - c_{V,V}$ also vanishes on the image of $1 + c_{V,V}$. We can also define quadratic forms:

\begin{Def}
    Let $V \in \mathcal{C}$ be an object in a symmetric tensor category $\mathcal{C}$ with unit object $\un$. A \textit{quadratic form} on $V$ is a map $\Gamma^2(V) \rightarrow \un$.
\end{Def}
The restriction $q|_W$ of a quadratic form $q$ on $V$ to a subobject $W$ is given by restricting to the subobject $\Gamma^2(W) \subseteq \Gamma^2(V)$. To see this definition coincides with the usual definition of a quadratic form when working over vector spaces, recall that if $q: V \rightarrow \mathbb{F}$ is a quadratic form, then 

\begin{equation}\label{coordinates}
    q(a_1e_1 + \cdots + a_ne_n) = \sum_{1\leq i \leq j \leq n} \lambda_{ij}a_ia_j,
\end{equation}
where $V$ has basis $\{e_1, \dots, e_n\}$ of $V$ and $\lambda_{ij} \in \mathbb{F}$ are suitable coefficients. Therefore, the space of quadratic forms is spanned by the elements of $S^2(V^*)$, which is isomorphic to $\Gamma^2(V)^*$. This definition also makes it clear that any symmetric bilinear form, given by a map $S^2(V) \rightarrow \un$, yields a quadratic form given by precomposing with the composition $\Gamma^2(V) \hookrightarrow V \otimes V \twoheadrightarrow S^2(V)$. In the case the underlying characteristic is not $2$, the decomposition $V \otimes V \cong S^2(V) \oplus \wedge^2(V)$ and the isomorphism $S^2(V) \cong \Gamma^2(V)$ implies there is a bijection between symmetric bilinear forms and quadratic forms. We will denote the associated symmetric bilinear form to a quadratic form $q$ by $\beta_q$ (in characteristic $2$ we will use the same notation for the associated form but a different definition, see \S \ref{char2considerations}). We are primarily interested when forms are \textit{non-degenerate}:

\begin{Def}
    A bilinear form $\beta : V \otimes V \rightarrow \un$ on an object $V$ in a symmetric tensor category $\mathcal{C}$ is \textit{non-degenerate} if the image of $\beta$, denoted $\beta'$, under the isomorphism 
    
    \[\Hom_{\mathcal{C}}(V\otimes V, \un) \rightarrow \Hom_{\mathcal{C}}(V, V^*)\]
    afforded by tensor-hom adjunction is an isomorphism itself. A quadratic form is \textit{non-degenerate} if the associated symmetric form is non-degenerate.
\end{Def}
We call the kernel of the map $\beta'$ the \textit{radical} of $\beta$.

\subsubsection{Additional considerations in characteristic $2$}\label{char2considerations}
Now suppose that we are working over a symmetric tensor category $\mathcal{C}$ in characteristic $2$ for this entire subsection. Then, for any $V \in \mathcal{C}$, the $S_2$-module $V\otimes V$ does not necessarily split, and $\wedge^2(V) = S^2(V)$. This means that the notion of a skew-symmetric form is redundant and that all alternating forms are symmetric. However, the reverse implication does not hold. Moreover, there is no longer a bijection between quadratic forms and symmetric forms; while every symmetric form still yields a quadratic form, it is not unique (alternating forms yield the zero quadratic form, for instance).
\par
In order to associate a symmetric form to a quadratic form, we need look at the structure of $V \otimes V$ in characteristic $2$. First, define the \textit{Frobenius twist} $V^{(1)}$ of an object $V$ in a symmetric tensor category to be the image of the following composition (first defined in \cite{coulembier2020tannakian}):

\begin{equation}\label{frobenius_twist}
    \Gamma^2(V) \hookrightarrow V \otimes V \twoheadrightarrow S^2(V).
\end{equation}
Let $A^2(V)$ be the kernel of the composite map. Then, $\Gamma^2(V)$ is an extension of its submodule $A^2(V)$ by $V^{(1)}$. Similarly, let $\mathbf{A}^2(V)$ denote the cokernel of the composite map, so that $S^2(V)$ is an extension of its submodule $V^{(1)}$ by $\mathbf{A}^2(V)$. Then, we have the following:

\begin{prop}\label{char2_exterior_powers_same}
    The following are true: 
    \[A^2(V) = \im(1 - c_{V,V});\]
    \[\mathbf{A}^2(V) \cong  (V \otimes V)/\ker(1- c_{V,V}),\]
    which implies that $A^2(V) \cong \mathbf{A}^2(V)$ by the first isomorphism theorem. 
\end{prop}

\begin{proof}
    Because $A^2(V)$ is the kernel of the map in \eqref{frobenius_twist}, it must be contained in $\im(1 - c_{V,V})$ by definition of $S^2(V)$. But then it is clear that $\im(1 - c_{V,V}) \subseteq \Gamma^2(V)$ because $(1-c_{V,V})^2 = 0$ in characteristic $2$, so $\im(1 - c_{V,V})$ lies in $A^2(V)$ as it is will map to zero under the map in \eqref{frobenius_twist}. This proves the first statement. 
    \par
    By the definition, $\mathbf{A}^2(V) = S^2(V)/V^{(1)}$. Then, notice that $S^2(V) = (V\otimes V)/\im(1 - c_{V,V})$ and $V^{(1)} = \Gamma^2(V)/\im(1 - c_{V,V})$. Apply the third isomorphism theorem to get the second statement.
\end{proof}
Notice that the Frobenius twist is the obstruction to identifying quadratic forms with symmetric bilinear forms. Nevertheless, this lets us make the following definition:

\begin{Def}
    Let $V$ be an object in an STC $\mathcal{C}$ over a field of characteristic $2$. Given a quadratic form $q: \Gamma^2(V) \rightarrow \un$, we can associate a bilinear form $\beta_q: V \otimes V \rightarrow \un$ to $q$ by first restricting to $A^2(V) \subset \Gamma^2(V)$, and then defining $\beta_q$ to be the composition 

    \[V \otimes V \twoheadrightarrow \mathbf{A}^2(V) \cong A^2(V) \xrightarrow{q} \un.\]
\end{Def}
Let's see how this generalizes what happens over vector spaces. Let $V$ be a vector space over a field of characteristic $2$ with basis $\{e_1, \dots, e_n\}$. Then, $A^2(V)$ has basis $\{e_i\otimes e_j + e_j\otimes e_i\}_{1  \leq i < j \leq n}$, which extends to a basis of $\Gamma^2(V)$ by including $\{e_i\otimes e_i\}_{1 \leq i \leq n}$. The space $\mathbf{A}^2(V)$ has basis $\{v_iv_j\}_{1 \leq i < j \leq n}$, and the isomorphism $A^2(V) \rightarrow \mathbf{A}^2(V)$ is given by sending $v_i\otimes v_j + v_j\otimes v_i$ to $v_iv_j$. Now, let $q: \Gamma^2(V) \rightarrow \un$ be a quadratic form, which we write in coordinates as in equation \eqref{coordinates}. It sends $v_i \otimes v_j + v_j \otimes v_i$ for $1 \leq i < j \leq n$ to $\lambda_{ij}$. This makes it clear that after the identification $A^2(V) \cong \mathbf{A}^2(V)$ the associated bilinear form $\beta_q: V \otimes V \rightarrow \un$ is given by $\beta_q(e_i, e_j) = \lambda_{ij}$ for $1 \leq i < j \leq n$ and $\beta_q(e_i, e_i) = 0$ for all $1 \leq i \leq n$. This agrees with the definition $\beta_q(v, w) = q(v+w) - q(v) - q(w)$ when working over vector spaces. We also have the following generalization:

\begin{prop}\label{restrict_beta_q}
    Let $q$ be a quadratic form on an object $V$ in $\mathcal{C}$. Then, $\beta_{q|_W} = \beta_q|_W$ for any subobject $W \subseteq V$.
\end{prop}
\begin{proof}
    The restriction of $\beta_q$ to $W \otimes W$ will not just pass through $V \otimes V/\Gamma^2(V)$ but also through $(\Gamma^2(V) + W\otimes W)/\Gamma^2(V)$. Therefore, $\beta_q|_W$ is given by the following composition: 

    \begin{align*}
        W \otimes W &\rightarrow (\Gamma^2(V) + W\otimes W)/\Gamma^2(V) \\ &\cong  W\otimes W/(W \otimes W \cap \Gamma^2(V)) \\ &\cong W \otimes W/\Gamma^2(W) \cong \mathbf{A}^2(W) \cong A^2(W) \rightarrow \un.
    \end{align*}
    But this is just $\beta_{q|_W}$.
\end{proof}

\par
Finally, it would be natural to have a notion of what it means for a quadratic form to be non-degenerate. However, even in the vector-space setting the definition does not appear to be standardized across the literature. For instance, given a quadratic form $q$ on $V$, some require that the associated symmetric form $\beta_q$ be non-degenerate, but this is appears to be too restrictive (for instance, there would be no non-degenerate quadratic forms on an odd-dimensional vector space). 
\par
Another definition, as in \cite{elman2008algebraic}, is that $q$ is non-degenerate if there is some algebraically closed field extension over which the radical $\mathfrak{r}_q \coloneqq \{v \in \mathfrak{r} \mid q(v) = 0\}$ of $q$ is just the zero vector. Here $\mathfrak{r}$ is the radical of the form associated to $q$. This definition affords a geometric characterization where the orthogonal group scheme associated to the quadratic form is reductive (simple when dimension larger than $4$), even over odd-dimensional spaces.
\par
It is unclear to us how to generalize this definition to the $\Ver_4^+$ setting, let alone to arbitrary symmetric tensor categories in characteristic $2$. This difficulty arises from the following facts. First of all, there are objects in $\Ver_4^+$ whose Frobenius twist is zero, which means that there is a bijection between quadratic and bilinear forms on them (and therefore it makes sense that a quadratic form is non-degenerate only if the associated form is non-degenerate). But the failure of the Frobenius functor to be exact means that these objects have subobjects whose Frobenius twist is non-zero. This means the radical of the form could have a non-zero Frobenius twist.
\par
Therefore, for simplicity, we will say $q$ is \textit{non-degenerate} if $\beta_q$ is non-degenerate.


\subsection{Operations on forms}\label{forms_operations}
In this subsection, we describe various operations on forms. Let us fix a symmetric tensor category $\mathcal{C}$ with braiding $c$ (we will always suppress associativity morphisms).

\begin{Def}
    Let $\beta$ and $\gamma$ be two bilinear forms objects $V,W \in \mathcal{C}$, respectively. We say that $\beta$ and $\gamma$ are \textit{equivalent} if there exists an isomorphism $\phi : V \rightarrow W$ such that $\gamma = \beta \circ (\phi \otimes \phi)$.
\end{Def}
It is obvious this forms an equivalence relation, and it is clear that non-degeneracy is preserved by this equivalence relation. We will denote the class of a bilinear form $\beta$ by $[\beta]$. We can also define the sum and product of forms:

\begin{Def}
    Let $\beta: V \otimes V \rightarrow \un$ and $\gamma: W \otimes W \rightarrow \un$ be two bilinear forms. Then, the \textit{sum} $\beta + \gamma: (V \oplus W) \otimes (V \oplus W) \rightarrow \un$ is the bilinear form on $V \oplus W$ given by the composition

    \[(V \oplus W) \otimes (V \oplus W) \cong (V\otimes V) \oplus (V\otimes W) \oplus (W \otimes V) \oplus (W \otimes W) \xrightarrow{\beta \oplus 0 \oplus 0 \oplus \gamma} \un \oplus \un \rightarrow \un, \]
    where the last map is the addition map.
\end{Def}
\begin{Def}
    Let $\beta: V \otimes V \rightarrow \un$ and $\gamma: W \otimes W \rightarrow \un$ be two bilinear forms. Then, the \textit{product} $\beta \times \gamma: (V \otimes W) \otimes (V \otimes W)\rightarrow \un$ is the bilinear form on $V \otimes W$ given by the composition

    \[(V \otimes W) \otimes (V \otimes W) \xrightarrow{1_V \otimes c_{W,V} \otimes 1_W} (V \otimes V) \otimes (W \otimes W) \xrightarrow{\beta \otimes \gamma} \un \otimes \un \xrightarrow{m} \un, \]
    where the last map $m$ is the isomorphism afforded by the unit object.
\end{Def}
It is easy to see that if $\beta$ and $\gamma$ are non-degenerate, then so are $\beta + \gamma$ and $\beta \times \gamma$. It is also clear that the sum of two symmetric forms is symmetric and that the sum of two skew-symmetric forms is skew-symmetric. The isomorphism $\Gamma^2(V \oplus W) \cong \Gamma^2(V) \oplus (V \otimes W) \oplus \Gamma^2(W)$ makes it clear that the sum of two alternating forms is also alternating. We have the following statement about the product of two forms, which is a consequence of the hexagonal axioms that the braiding must satisfy:

\begin{prop}\label{products}
    Let $\beta: V \otimes V \rightarrow \un$ and $\gamma: W \otimes W \rightarrow \un$ be two bilinear forms. If $\beta$ and $\gamma$ are both symmetric or are both skew-symmetric, then $\beta \times \gamma$ is symmetric. If one form is symmetric and the other is skew-symmetric, then $\beta \times \gamma$ is skew-symmetric.
\end{prop}
\begin{proof}
    Deferred to Appendix \S\ref{appendix}.
\end{proof}
We can define the sum and product of the isomorphism classes of two bilinear forms $[\beta]$ and $[\gamma]$ by taking suitable representatives $\beta$ and $\gamma$, taking their sum or product, and then taking the isomorphism class of the result. It is clear that this is well-defined. Now, let $\mathcal{W}(\mathcal{C})$ denote the set of isomorphism classes of non-degenerate symmetric bilinear forms. Our results thus far show that $\mathcal{W}(\mathcal{C})$ is closed under sum and product. We actually can say more:

\begin{prop}\label{witt_semi_ring}
    The set $\mathcal{W}(\mathcal{C})$ equipped with $(+, \times)$ is an associative, commutative semi-ring with additive identity $[0]$ and multiplicative identity $[m]$, where $0$ is the zero form on the zero object. We call this semi-ring the \textit{Witt semi-ring} of $\mathcal{C}$.
\end{prop}
\begin{proof}
    Deferred to Appendix \S\ref{appendix}.
\end{proof}
Now, let us turn to operations on quadratic forms. 

 \begin{Def}
 Let $q$, $r$ be two quadratic forms on $V$ and $W$, respectively. We say $q$ and $r$ are equivalent if there exists an isomorphism $\phi: V \rightarrow W$ such that the induced map $\Gamma^2(\phi): \Gamma^2(V) \rightarrow \Gamma^2(W)$ satisfies $q = r \circ \Gamma^2(\phi)$.
 \end{Def}
 Again, this defines an equivalence relation on the space of quadratic forms which is stable under restricting to non-degenerate quadratic forms. Let us use $\mathcal{Q}(\mathcal{C})$ to denote the set of isomorphism classes of non-degenerate quadratic forms in $\mathcal{C}$. 

 \begin{Def}
    Given two quadratic forms $q$ on $V$ and $r$ on $W$, their sum $q+r$ on $V \oplus W$ is given by the composition

    \[\Gamma^2(V \oplus W) \cong \Gamma^2(V) \oplus V \otimes W \oplus \Gamma^2(W) \xrightarrow{q \oplus 0 \oplus r} \un \oplus \un \rightarrow \un,\]
    where the last map is addition. 
 \end{Def}

 \begin{Def}
    Given a bilinear form $\gamma$ on $V$ and a quadratic form $q$ on $W$, we can produce a new quadratic form $\gamma.q$ on $V \otimes W$ by the composition

    \[\Gamma^2(V \otimes W) \hookrightarrow V \otimes W \otimes V \otimes W \xrightarrow{1_V \otimes c_{W,V} \otimes 1_W} V \otimes V \otimes W \otimes W \xhookrightarrow{\gamma \otimes \beta_q} \un \otimes \un \xrightarrow{m} \un.\]
 \end{Def}
 It is clear that these operations descend to $\mathcal{Q}(\mathcal{C})$ in the sense that $\mathcal{Q}(\mathcal{C})$ is a semi-module over the semi-ring $\mathcal{W}(\mathcal{C})$. The proof of this is similar to the proofs of Proposition \ref{products} and Proposition \ref{witt_semi_ring} and is omitted. It is also clear that this enables one to define the multiplication of two quadratic forms, but given that there is no unit element in characteristic $2$ for such a multiplication, we do not consider it further. Outside of characteristic $2$ it just coincides with the Witt semi-ring.

 \begin{Question}
    Given an arbitrary symmetric tensor category $\mathcal{C}$, what can we say about the structure of its Witt semi-ring? In characteristic $2$, what can we say about the semi-module structure of $\mathcal{Q}(\mathcal{C})$ over $\mathcal{W}(\mathcal{C})$?
 \end{Question}
A partial answer to this question is easy for any Frobenius-exact symmetric tensor category, as such a category would fiber over the Verlinde category $\Ver_p$. Over $\Ver_p$, because each simple object is self-dual, it is easily seen that every non-degenerate symmetric bilinear form consists of the information of $\tfrac{p-1}{2}$ ordinary non-degenerate symmetric bilinear forms over vector spaces and $\tfrac{p-1}{2}$ ordinary non-degenerate skew-symmetric bilinear forms over vector spaces. This corresponds to the decomposition $\Ver_p = \Ver_p^+ \boxtimes \sVec_\KK$ and the fact that the odd generator $L_{p-1}$ of $\sVec_\KK$ satisfies $S^2(L_{p-1}) = \un$.

\section{Basic Properties of the Verlinde Category \texorpdfstring{$\ver_4^+$}{}}\label{basic_properties_section}
In this section, we define the Verlinde category $\Ver_4^+$ and state its basic properties. Throughout this paper, we define $\KK$ as an algebraically closed field of characteristic $p = 2$. We will also assume a cursory familiarity with the language of Hopf algebras and tensor categories (cf.\@ \cite{etingof2016tensor, etingof2021lectures}) and suppress associativity morphisms in our notation.

\subsection{The Hopf Algebra $\KK[t]/(t^2)$}
The unital algebra $A \coloneqq \KK[t]/(t^2)$ admits the structure of a Hopf algebra with comultiplication $\Delta : A \rightarrow A \otimes A$, counit $\epsilon: A\rightarrow \KK$, and antipode $S: A \rightarrow A$ uniquely determined by 

\begin{align*}
    \Delta(t) = 1 \otimes t + t \otimes 1; \\
    \epsilon(t) = 0; \\
    S(t) = t.
\end{align*}
By the theory of Jordan canonical forms, $A$ has two indecomposable modules up to isomorphism: the trivial representation, denoted $\un$, which is simple, and a two-dimensional module $P$, which is an extension of $\un$ by itself. The Krull-Schmidt theorem tells us that any module $U$ over $A$ is (non-uniquely) isomorphic to $m\un \oplus nP$, with $m$ and $n$ invariants of $U$. We will often fix such a decomposition and let the sets 
\begin{equation}\label{standard_basis}
\begin{aligned}
    \{v_1, v_2, \dots, v_m\} \\
    \{w_1, x_1, \dots, w_n, x_n\}       
\end{aligned}
\end{equation}
denote a basis of $m\un$ and a basis of $nP$, respectively, where $t.v_j = 0$ for all $1 \leq j \leq m$ and $t.w_k = x_k$ for all $1 \leq k \leq n$. Notice in particular that $t.x_k = 0$ as $t^2 = 0$. We have $U = V \oplus W \oplus X$, where $V$ is the span of the vectors $\{v_j\}_{j=1}^m$, $W$ is the span of the vectors $\{w_k\}_{k=1}^n$, and $X$ is the span of the vectors $\{x_k\}_{k=1}^n$. The vector space of morphisms $\Hom_A(M, N)$ between two representations $M,N$ is simply the collection of linear maps that respect the $t$-action, meaning that $t.\phi(\mu) = \phi(t.\mu)$ for all $\mu \in M$ and $\phi \in \Hom_A(M,N)$.

Note that the linear map $\varphi \in \Hom_{\ver_4^+}(U, U)$ given by $\varphi(u)=t.u$ is a morphism in the category $\ver_4^+$ because it commutes with the $t$-action. With respect to the decomposition of $U$ described above, $\im(\varphi)=X$ and $\ker(\varphi)=V\oplus X$. Thus, $X$ and $V\oplus X$ are fixed, while $V$ and $W$ are dependent on a choice of basis because the decomposition of $U$ into $m\un\oplus nP$ is not unique.
\par
Given an $A$-module $U$, there is a (left) dual module $U^*$ with the $t$-action defined by
\[ (t.f)(u) = f(S(t).u) = f(t.u) \]
for all $f \in U^*$. With respect to the basis of $U$ given by \eqref{standard_basis}, $U^*$ has a dual basis given by the union of the following two sets:
\begin{equation}\label{standard_dual_basis}
\begin{aligned}
    \{v_1^*, v_2^*, \dots, v_m^*\} \\
    \{x_1^*, w_1^*, \dots, x_n^*, w_n^*\}.       
\end{aligned}
\end{equation}
Here, $t.v_j^* = 0$ for all $1 \leq j \leq m$, and $t.x_k^* = w_k^*$ for all $1 \leq k \leq n$. Finally, given any two $A$-modules $M$ and $N$, the tensor product $M\otimes N$ admits the structure of an $A$-module via the comultiplication map. It is determined by 
\[t.(\mu \otimes \nu) = (t.\mu)\otimes \nu + \mu \otimes (t.\nu)\]
for all $\mu \in M$ and $\nu \in N$. Explicitly, if two copies of $P$ have a fixed bases $\{w, x\}$ and $\{\omega, \chi\}$, respectively, then their tensor product is $P\otimes P = P\oplus P$. A basis for the first summand is $\{w \otimes \chi, x \otimes \chi\}$, and a basis for the second summand is $\{w\otimes \omega, x \otimes \omega + w \otimes \chi\}$.
\par
We can then define the representation category $\Rep A$ to be the category whose objects are $A$-modules and whose morphisms between two $A$-modules $M,N$ are the maps $\Hom_A(M, N)$. These structures endow $\Rep A$ with the structure of a tensor category. 
\subsection{Triangular Structure on $\KK[t]/(t^2)$ and the Verlinde Category $\Ver_4^+$}
The Hopf algebra $A$ is said to have a \textit{triangular structure} \textit{with} $R$-\textit{matrix} $R$ if there exists an invertible element $R$ in the algebra $A \otimes A$ such that the following identities hold:
\begin{equation}\label{triangular_structure_axioms}
    \begin{aligned}
        (\Delta \otimes 1_A)(R) &= R^{13}R^{23}; \\
        (1_A \otimes \Delta)(R) &= R^{13}R^{12}; \\
        (\sigma_{A,A} \circ \Delta)(a) &= R \Delta(a) R^{-1} \ \ \forall a \in A; \\
        R^{-1} &= R^{21},
    \end{aligned}
\end{equation}

where $\sigma_{X,Y}$ is the permutation of components on $X \otimes Y$. The term $R^{i_1,\dots, i_k}$ is given by permuting $R \otimes 1^{l-2}$ so that the component of $R$ along the $j$-th tensor is now along the $i_j$-th component and where the value of $l$ is determined by the number of tensors on the left-hand side. For example, if $R = a\otimes b  + c \otimes d$ and $l = 3$, then $R^{13} = a \otimes 1 \otimes b + c \otimes 1 \otimes d$. Given a triangular structure on $A$, we can endow $\Rep A$ with a symmetric structure to construct the symmetric tensor category $\Rep (A, R)$. We define the braiding $c$, a natural transformation between the bifunctors $- \otimes - : \Rep A \times \Rep A \rightarrow \Rep A$ and $\sigma_{\_,\_} \circ (-\otimes -): \Rep A \times \Rep A \rightarrow \Rep A$, by  
\[c_{V,W}(v\otimes w) = \sigma_{V,W}(R.(v \otimes w)) \]
for all $V,W \in \Rep A$ and $v \in V, w \in W$. In the case $R = 1 \otimes 1$, we recover the usual symmetric structure on the category $\Rep A$. 
\begin{lem}\label{triangular_structure}
There is a triangular structure on $A$ with $R$-matrix given by $R = 1 \otimes 1 + t \otimes t$.
\end{lem}

\begin{proof}
    Deferred to Appendix \S\ref{appendix}.
\end{proof}
Therefore, we have the following definition:

\begin{Def}
    The \textit{Verlinde category} $\Ver_4^+$ is the representation category $\Rep (A, R)$, where $A = \KK[t]/(t^2)$ and $R = 1\otimes 1 + t \otimes t$ is the $R$-matrix  imposing the triangular structure on $A$.
\end{Def}
The braiding $c$ is explicitly given by 
\[c_{V,W}(v\otimes w) = w \otimes v + (t.w)\otimes (t.v)\]
for all $V,W \in \Rep A$ and $v \in V, w \in W$. It is shown in \cite{venkatesh_hilbert_basis} that $\Ver_4^+$ does not fiber over the category of vector spaces \footnote{There is no category of supervector spaces in characteristic $2$, but in loc.\@ cit.\@, it is suggested that $\Ver_4^+$ could be viewed as the analog in characteristic $2$.}. For more information on triangular Hopf algebras, see \cite[\S8.3]{etingof2016tensor}. Because the underlying tensor category of $\Ver_4^+$ is $\Rep A$, we can and will use the language of vector spaces to describe objects and maps.

\subsection{Properties about Forms in $\Ver_4^+$}
For the remainder of this paper we will work with $\mathcal{C} = \Ver_4^+$. From now on, we will freely identify a bilinear form $\beta$ on an object $U$ with the corresponding bilinear map $U \times U \rightarrow \un$, so we sometimes write $\beta(u, u')$ instead of $\beta(u \otimes u')$. Let us now try to understand forms in this category a bit better. First of all, a bilinear form $\beta: U \otimes U \rightarrow \un$ must satisfy
\begin{align*}
    0 &= t.(\beta(u \otimes u')) = \beta(t.(u \otimes u'))\\
    &= \beta((t.u)\otimes u' + u \otimes (t.u')) \\
    &\implies \beta(t.u \otimes u') = \beta(u \otimes t.u')
\end{align*}
for all $u, u' \in U$, because $\beta$ is also an $A$-module homomorphism. 
\par
Our first major property about symmetric bilinear forms in $\Ver_4^+$ is that they reduce to symmetric bilinear forms in the underlying category $\Rep A$:

\begin{lem}\label{symcondition}
Let $\beta: U \otimes U \rightarrow \un$ be a bilinear form in $\Ver_4^+$. Then, $\beta$ is symmetric if and only if  $\beta(u\otimes u') = \beta(u'\otimes u)$ for all $u, u' \in U$.
\end{lem}

\begin{proof}
Suppose $\beta$ is symmetric. Then,
\begin{align*}
    \beta(u \otimes u') &= \beta(u' \otimes u) + \beta((t.u') \otimes (t.u)) \\
    &= \beta(u' \otimes u) + \beta(u' \otimes (t^2.u)) \\
    &= \beta(u' \otimes u).
\end{align*}
The reverse direction follows by running these steps backwards.
\end{proof}
Therefore, we can think of classification of non-degenerate symmetric bilinear forms in $\Ver_4^+$ as follows: we view the form $\beta$ as a symmetric bilinear form on an ordinary vector space $U$ and then we classify the nilpotent orbits in the Lie algebra of the orthogonal group preserved by $\beta$ satisfying $t^2 = 0$. A similar problem has been studied in \cite{xue2009nilpotent}. It will also be useful to have an explicit description of the second-divided power of an arbitrary object in $\Ver_4^+$. 

\begin{prop}\label{gamma2_basis_prop}
    Let $U$ be an object in $\Ver_4^+$ with a decomposition $U = m\un \oplus nP = V \oplus W \oplus X$ as in \ref{standard_basis}. Then, a decomposition of $\Gamma^2(U)$ into indecomposables is given by 
    
    \begin{equation}\label{gamma2_basis}
    \begin{aligned}
        v_{i'} \otimes v_{i'} &\mapsto 0 \\
        v_i \otimes v_j + v_j \otimes v_i &\mapsto 0 \\
        v_{i'} \otimes w_{k'} + w_{k'} \otimes v_{i'} \mapsto  v_{i'} \otimes x_{k'} + x_{k'} \otimes v_{i'} &\mapsto 0 \\
        x_{k'} \otimes x_{k'} &\mapsto 0 \\
        w_{k'} \otimes x_{k'} + x_{k'} \otimes w_{k'} &\mapsto 0 \\
        w_k \otimes x_l + x_l \otimes w_k \mapsto x_k \otimes x_l + x_l \otimes x_k &\mapsto 0 \\
        w_k \otimes w_l + w_l \otimes w_k + x_k \otimes x_l \mapsto x_k \otimes w_l + w_k \otimes x_l + x_l \otimes w_k + w_l \otimes x_k &\mapsto 0.
    \end{aligned}
    \end{equation} 
    where each line denotes an indecomposable spanned by the given vector(s) and where the $\mapsto$ symbol denotes the image of a basis vector under the action of $t$. The variable ranges are given by $1 \leq i' \leq m$, $1 \leq i < j \leq m$, $1 \leq k' \leq n$, and $1 \leq k < l \leq n$.
\end{prop}

\begin{proof}
    The proof is basically identical to how one goes about finding a basis for $\Gamma^2(Z)$ for an ordinary vector space $Z$, except there is a subtlety that arises when dealing with the vectors of the form $w_{k'} \otimes w_{k'}$ and $w_k \otimes w_l + w_l \otimes w_k$, which is the only instance where the braiding on $\Ver_4^+$ differs from the usual braiding on $\Rep A$. The modification corresponds to the last line of \eqref{gamma2_basis}.
\end{proof}
Notice now that any quadratic form $q$ is determined by its values on the left-most vector of each line of \eqref{gamma2_basis}. We have the following corollary.

\begin{cor}\label{frobtwist0}
    The Frobenius twist of $nP$ is $0$, which means that there is a bijection between (non-degenerate) symmetric bilinear forms on $nP$ and (non-degenerate) quadratic forms on $nP$.
\end{cor}
\begin{proof}
    The Frobenius twist of $nP$ is by definition the image of the canonical map $\Gamma^2(nP) \twoheadrightarrow S^2(nP)$. The last four lines of \eqref{gamma2_basis} are the basis vectors for $\Gamma^2(nP)$. It is clear that all of these lie in the image of $1-c_{V,V}$ (in particular, $x_{k'} \otimes x_{k'}$ is the image of $w_{k'} \otimes w_{k'}$ and $w_k \otimes w_l + w_l \otimes w_k + x_k \otimes x_l$ is the image of $w_l \otimes w_k$). The discussion preceding Proposition \ref{char2_exterior_powers_same} then makes it clear that $\Gamma^2(nP) \cong S^2(nP)$, which gives us a way to identify symmetric bilinear forms on $nP$ with quadratic forms on $nP$.
\end{proof}

We can also identify the additional criteria that symmetric bilinear forms must satisfy to be alternating.

\begin{prop}\label{alternatingfirstdef}\label{alternatingmaindef}
    Let $\beta: U \otimes U \rightarrow \un$ be a symmetric bilinear form in $\Ver_4^+$. Then, $\beta$ is alternating if and only if $\beta(u \otimes u) = 0$ for all $u \in \ker t$. 
\end{prop}

\begin{proof}
    Recall that if we fix a decomposition $U = V \oplus W \oplus X$ as in \ref{standard_basis}, then $\ker t = V \oplus X$. Now, suppose that $\beta$ is an alternating form. Then, it is clear by Proposition \ref{gamma2_basis_prop} that $\beta(u \otimes u) = 0$ for all $u \in \ker t$ (the same proof for vector spaces goes through). 
    \par
    On the other hand, suppose that $\beta(u \otimes u) = 0$ for all $u \in \ker t$. It is clear that each basis vector in \eqref{gamma2_basis} is killed by $\beta$ (recall $\beta$ is also symmetric by assumption), except vectors of the form $w_k \otimes w_l + w_l \otimes w_k + x_k \otimes x_l$. But these also map to zero because $\beta(w_k \otimes w_l + w_l \otimes w_k) = 0$ by symmetry and $\beta(x_k \otimes x_l) = 0$ because $\beta(x_k \otimes x_l) = \beta(t.(w_k \otimes x_l)) = t.\beta(w_k \otimes x_l) = 0$ as $t$ acts trivially on $\un$.
\end{proof}
Note that any symmetric bilinear form on $nP$ is alternating as in this case $V = 0$ and $\beta|_{X\otimes X} = 0$ (see Lemma \ref{bilform_calcs}). The following terminology will be useful later on:

\begin{Def}
    Let $\beta$ be a symmetric bilinear form on $U$. Call $\beta$ \textit{super-alternating} if $\beta(u \otimes u) = 0$ for all $u \in U$.
\end{Def}
Clearly any super-alternating form is alternating, and the two notions coincide when working over vector spaces.
\par
As in the ordinary vector space setting, decomposing a bilinear form into the sum of smaller forms by way of orthogonal complements will be a key idea. If $\beta$ is a bilinear form on $U$ and $S$ is a subobject of $U$, we define the \textit{orthogonal complement} $S^{\perp}$ of $S$ (in $U$ and with respect to $\beta$) to be
\[S^\perp \coloneqq \ker(U \xrightarrow{\beta'} U^* \xrightarrow{\pi} S^*),\]
where the map $\pi$ is the usual projection map. The following proposition is useful:

\begin{prop}\label{nondegenperp}
    Let $\beta$ be a non-degenerate symmetric bilinear form on $U \in \Ver_4^+$, and let $S$ be a subobject of $U$. If the restriction of $\beta$ to $S$ is non-degenerate, then $U = S \oplus S^\perp$, and moreover, the restriction of $U$ to $S^{\perp}$ is also non-degenerate.
\end{prop}

\begin{proof}
    The proofs in the classical setting extend to our setting (\cite[Theorem 3.12]{Conrad2008BILINEARF}).
\end{proof}



\section{Classification of Non-Degenerate Forms in \texorpdfstring{$\ver_4^+$}{}}\label{classification}
We have now set the stage to classify the isomorphism classes of non-degenerate symmetric bilinear forms and non-degenerate quadratic in $\Ver_4^+$. From now on, whenever we say two bilinear forms are equal to each other, we always mean up to isomorphism. 

\subsection{Classifying forms on objects of the form $m\un$ and of the form $nP$}
Before we can approach the general case, it is easier to classify forms on objects of the form $m\un$ and on objects of the form $nP$. The former is the well-known classification of symmetric bilinear forms in the ordinary vector space setting:

\begin{thm}[\cite{glasser05}]\label{vforms}
    Let $\beta$ be a non-degenerate symmetric bilinear form on a vector space $Z$. Then, there exists a basis for $Z$ in which the associated matrix of $\beta$ is either the identity matrix or direct sums of the $2 \times 2$ matrix given by 
    \[ \begin{bmatrix} 0 & 1 \\ 1 & 0 \end{bmatrix}.\]
    For each dimension, the corresponding isomorphism classes of these two forms are non-isomorphic. If $\dim Z = m$, let us denote the first form by $\alpha_1^m$ and the second form by $\alpha_2^m$ (which exists only for even $m$).
\end{thm}
Changing basis amounts to conjugation by an invertible map in $\Hom(Z,Z)$. However, the endomorphism spaces in $\Ver_4^+$ are considerably more restrictive, and therefore, we find more isomorphism classes of non-degenerate symmetric bilinear forms. We start our classification with the following straightforward lemma:

\begin{lem}\label{bilform_calcs}
Let  $\beta$ a symmetric bilinear form on an object $U \in \Ver_4^+$ with the decomposition $U = m\un \oplus nP = V \oplus W \oplus X$ arising from the basis described by \eqref{standard_basis}. Then, $\beta$ must satisfy the following for all $1 \leq i \leq m$ and $1 \leq j,k \leq n$:

\begin{enumerate}
    \item $\beta(v_i, x_j) = 0$, meaning $\beta|_{V \otimes X} = 0$ and $\beta|_{X \otimes V} = 0$;
    \item $\beta(w_j, x_k) = \beta(x_j, w_k)$;
    \item $\beta(x_j, x_k) = 0$, meaning $\beta|_{X \otimes X} = 0$.
\end{enumerate}
\end{lem}
\begin{proof}
    This is a direct consequence of the fact that $\beta(t.u, u') = \beta(u, t.u')$ for all $u,u' \in U$.
\end{proof}
The following motivates why we first consider the classification of $m\un$ and $nP$ separately.

\begin{prop}\label{splitvfromu}
Let  $\beta$ a non-degenerate symmetric bilinear form on an object $U \in \Ver_4^+$ with the decomposition $U = m\un \oplus nP = V \oplus W \oplus X$ arising from the basis described by \eqref{standard_basis}. Then, the restriction of $\beta$ to $V$ is also non-degenerate. 
\end{prop}

\begin{proof}
    Suppose for the sake of contradiction that $\beta$ is degenerate on $V$. Then, there exists a nonzero vector $v \in V$ such that $\beta|_{\KK v \otimes V} = 0$. By Lemma ~\ref{bilform_calcs}, we know that $\beta|_{\KK v \otimes X} = 0$ and $\beta|_{X \otimes (V\oplus X)} = 0$. Therefore, $\beta|_{(\KK v \oplus X) \otimes (V \oplus X) } = 0$, and the adjunct map $\beta': U \rightarrow U^*$ must map any $u\in \KK v\oplus X$ to a vector in $W^*$, where we decompose $U^* = V^* \oplus W^* \oplus X^*$. However, $\dim(\KK v\oplus X)=n+1$ and $\dim(W^*) = n$, so there exists a nonzero vector $u$ in $\KK v\oplus X$ such that $\beta'(u)=0$, contradicting the non-degeneracy of $\beta$ on $U$.
\end{proof}

An object $U \in \ver_4^+$ can be decomposed into $V \cong m\un$ and $V^\perp \cong nP$. If $\beta$ is a non-degenerate symmetric bilinear form on $U$, then by Propositions ~\ref{splitvfromu} and ~\ref{nondegenperp}, we can choose $V$ such that both $\beta|_V$ and $\beta|_{V^\perp}$ are non-degenerate symmetric bilinear forms. Because $V$ is an ordinary vector space, we already know that $\beta|_V$ belongs to one of the two classes in Theorem ~\ref{vforms}. In the remainder of this section, we will classify isomorphism classes of forms on $V^\perp \cong nP$.
\par
We will first show that on the object $P$, there exist infinitely many isomorphism classes of bilinear forms, each indexed by an element of $\KK$. We will denote suitable representatives for these isomorphism classes as $\beta_P(y): P \otimes P \rightarrow \un$, where $y \in \KK$. Similarly, on the object $2P$, there exist two isomorphism classes not arising from $\beta_P(y) + \beta_P(z)$, which we will call $\beta_{2P}(i): 2P \otimes 2P \rightarrow \un$ for $i = 0, 1$. 

\begin{lem}\label{p_forms}
    Let $\eta$ be a non-degenerate symmetric bilinear form on the object $P$. There exists a basis of $P$ such that the associated matrix of $\eta$ is given by 
    \begin{equation}
        \begin{bmatrix} y & 1 \\ 1 & 0 \end{bmatrix}
    \end{equation} 
    for suitable $y \in \KK$. These forms are pairwise non-isomorphic.
\end{lem}

\begin{proof}
    Let $p, q$ be basis vectors of $P$ such that $t.p = q$. The quantity $\eta(p, q)$ is nonzero as otherwise, $q$ would be in the kernel of $\eta$, and the form would be degenerate. Moreover, $\eta(q,q) = \eta(t.p, t.p)=\eta(p, t^2.p)=0$. Therefore, we can scale the basis vectors by $1 / \sqrt{\eta(p, q)}$ (which is a valid base change), and the associated matrix of $\eta$ with respect to this new basis is given by 
     \[ \begin{bmatrix} \tfrac{\eta(p,p)}{\eta(p,q)} & 1 \\ 1 & 0 \end{bmatrix}.\]
     Now, any map $P \rightarrow P$ is determined by where it sends $p$, so it follows immediately that these forms are pairwise non-isomorphic.
\end{proof}
The isomorphism class arising from the form in Lemma ~\ref{p_forms} will be represented by $\beta_P(y)$ for $y \in \KK$.  We can also classify some forms on the object $2P$.

\begin{Def}
    We say a symmetric bilinear form $\beta$ on an object $U \in \ver_4^+$ is $\textit{oscillating}$ if for all $u \in U$, we have $\beta(u, t.u) = 0$.
\end{Def}
With this definition, we have the following lemma:

\begin{lem}\label{2p_forms}
    Let $\eta$ be a non-degenerate oscillating bilinear form on the object $nP$ (with $n > 1$). Then, there is a subobject $S \cong 2P$ of $nP$ such that the restriction of $\eta$ to $S$ is non-degenerate, and moreover, there exists a basis of $S$ for which the associated matrix of $\eta|_S$ is given by one of the following two matrices:
    \begin{equation}\label{2p_form0}
        \left[\begin{matrix}
            0 & 0 & 0 & 1\\
            0 & 0 & 1 & 0\\
            0 & 1 & 0 & 0\\
            1 & 0 & 0 & 0\\
        \end{matrix}\right],
    \end{equation}
     \begin{equation}\label{2p_form1}
         \left[\begin{matrix}
    		1 & 0 & 0 & 1\\
    		0 & 0 & 1 & 0\\
    		0 & 1 & 0 & 0\\
            1 & 0 & 0 & 0\\
        \end{matrix}\right].
    \end{equation}
    The first form will be denoted as $\beta_{2P}(0)$, and the second will be denoted as $\beta_{2P}(1)$. These forms are not isomorphic (and are also not isomorphic to $\beta_P(y) + \beta_P(z)$ for any $y,z \in \KK$).
\end{lem}

\begin{proof}
    Let $p$ be a vector in $nP$ such that $t.p \neq 0$ (such a vector necessarily exists). The non-degeneracy of $\eta$ means there must exist a vector $q \in nP$ such that $\eta(t.p, q) \neq 0$. By the assumption that $\eta$ is oscillating, $\eta(u, t.u) = 0$ for all $u \in nP$. Therefore, $q \neq p$. Since $0 \neq \eta(t.p, q) = \eta(p, t.q)$, we have $t.q \neq 0$. Let $S$ be the subobject of $nP$ spanned by the basis vectors $\{p,t.p,q,t.q\}$. The matrix associated to $\eta|_S$ on this basis is of the form
    \begin{equation}\notag
            \left[\begin{matrix}
              \ast & 0 & \ast & \lambda \\
              0 & 0 & \lambda & 0\\
              \ast & \lambda & \ast & 0\\
            \lambda & 0 & 0 & 0\\
            \end{matrix}\right]
    \end{equation}
    for some nonzero $\lambda \in \KK$ and with $\ast$ denoting suitable entries such that the matrix is symmetric. Once we rescale each basis vector by $\frac{1}{\sqrt{\lambda}}$, the matrix with respect to this basis becomes 
    \begin{equation*}
            \left[\begin{matrix}
              b & 0 & c & 1\\
              0 & 0 & 1 & 0\\
              c & 1 & a & 0\\
                1 & 0 & 0 & 0\\
            \end{matrix}\right],
    \end{equation*}
    where $a, b, c \in \KK$. Then, we replace $q$ by $q'=q + c(t.q)$, which is a valid change of basis because $t.(q + c(t.q)) = t.q$. The associated matrix of $\eta$ is now given by
    \begin{equation*}
    \begin{bmatrix}
        b & 0 & 0 & 1\\
        0 & 0 & 1 & 0 \\
        0 & 1 & a & 0 \\
        1 & 0 & 0 & 0
    \end{bmatrix}.
    \end{equation*}
    The matrix above has determinant $1$, so this basis change preserves non-degeneracy. 
    
    If $a = b = 0$, we get the isomorphism class $\beta_{2P}(0)$, as claimed. Now, suppose $b \neq 0$ but $a = 0$. We can define $p' = \tfrac{1}{\sqrt{b}}p$ and $q'' = \sqrt{b}q'$. Then, with respect to the basis $\{p', t.p', q'', t.q''\}$, the associated matrix of $\eta$ is given by 
    \begin{equation*}\label{eqn:100}
    \begin{bmatrix}
        1 & 0 & 0 & 1\\
        0 & 0 & 1 & 0 \\
        0 & 1 & 0 & 0 \\
        1 & 0 & 0 & 0
    \end{bmatrix},
    \end{equation*}
    which is the associated matrix of the form $\beta_{2P}(1)$ representing our second isomorphism class. Similarly, if $a = 0$ and $b \neq 0$, we can interchange the order of $p, t.p$ with $q',t.q'$ in our basis and then apply the same process, which will give us the same matrix. Therefore, suppose that both $a$ and $b$ are nonzero. We can find $d\in \KK$ such that $k \coloneqq \sqrt{b} + d\sqrt{a} \neq 0$. We define a new basis $\{p', t.p', q'', t.q''\}$ of $2P$ given by $p'=\frac{1}{k}(p + dq' + da(t.p) + b(t.q'))$ and $q''=\sqrt{a}p+\sqrt{b}q'$. We have:
    \begin{itemize}
        \item $\eta(p',p')=\frac{1}{k^2}(b + d^2a)=\frac{1}{k^2}(k^2)=1$,
        \item $\eta(p',t.p')=\frac{1}{k^2}(2d)=0$,
        \item $\eta(p',q'')=\frac{1}{k}(\sqrt{w}y+b\sqrt{y}w+\sqrt{a}b+d\sqrt{b}a)=0$,
        \item $\eta(t.p', q'')=\frac{1}{k}(\sqrt{b} + d\sqrt{a})=\frac{1}{k}(k)=1$, and
        \item $\eta(q'',q'')=(\sqrt{a})^2b+(\sqrt{b})^2a= 0$.
    \end{itemize}
    Therefore, with respect to this new basis, the associated matrix of the form is
        \begin{equation*}
        \begin{bmatrix}
        1 & 0 & 0 & 1\\
        0 & 0 & 1 & 0 \\
        0 & 1 & 0 & 0 \\
        1 & 0 & 0 & 0
        \end{bmatrix},
        \end{equation*}
        which we have already seen. Thus, we obtain the form $\beta_{2P}(0)$ when $y = w = 0$ and $\beta_{2P}(1)$ otherwise. 
        \par
        To see that $\beta_{2P}(0)$ and $\beta_{2P}(1)$ give rise to distinct isomorphism classes, notice that the first form is super-alternating and the second form is not. Moreover, these two forms are oscillating, so they are non-isomorphic to the forms $\beta_P(k) + \beta_P(l)$ where $k, l \in \KK$, which are not oscillating.
\end{proof}
The forms arising in Lemma ~\ref{p_forms} and Lemma ~\ref{2p_forms} serve as the building blocks for all forms on $nP$, as the next lemma demonstrates.

\begin{lem}\label{splitp}
Any non-degenerate symmetric bilinear form $\beta$ on the object $nP$ admits one of the following sum decompositions:
\begin{align*}
    \beta &= \sum_{i=1}^{n} \beta_P(y_i) \\
    \beta &= \sum_{j=1}^{n/2} \beta_{2P}(a_j)
\end{align*}
for suitable $y_i \in \KK$ and $a_j \in \{0, 1\}$.
\end{lem}

\begin{proof}
Suppose that we can find a vector $u \in nP$ such that $\beta(u, t.u) \neq 0$. Then, $\beta$ restricted to the subobject $Z$ of $nP$ spanned by $\{u, t.u\}$ is non-degenerate, and therefore, by Lemma ~\ref{p_forms}, $\beta|_Z$ is in the isomorphism class as $\beta_{P}(y)$ for some $y \in \KK$. 

Otherwise, we have $\beta(u, t.u) = 0$ for all $u \in nP$ (i.\@e.\@ the form is oscillating). In this case, Lemma ~\ref{2p_forms} applies, and we can find a subobject $Y$ of $nP$ for which the restriction of $\beta$ gives the form $\beta_{2P}(a_j)$. 

In either case, once we find such a subobject $Z$ or $Y$, we can take its orthogonal complement and proceed inductively by way of Proposition ~\ref{nondegenperp}. This proves that $\beta$ is of the form
\[ \beta = \sum_i \beta_P(y_i) + \sum_j \beta_{2P}(a_j) \]
for suitable $y_i \in \KK$ and $a_j \in \{0, 1\}$. Now, given this decomposition, suppose that both isomorphism classes are present. Then, there is a basis $\{p, t.p, q, t.q, r, t.r\}$ of a subobject $S \cong 3P$ of $nP$ such that the associated matrix of $\beta|_S$ relative to this basis is given by 
\[ \begin{bmatrix}
    y & 1 & & & & \\
    1 & 0 &   &  &  &  \\
    & & a  & 0 & 0 & 1 \\
    & & 0  & 0 & 1 & 0 \\
    & & 0  & 1 & 0 & 0 \\
    & & 1  & 0 & 0 & 0
\end{bmatrix},\]
with $y \in \KK$ and $a \in \{0, 1\}$. Let $p' = p + q + r + (y+a)(t.p)$ and $q' = p + q$. Then, let $\tilde{S}$ denote the subobject of $S$ spanned by $\{p', t.p', q', t.q'\}$. With respect to this basis, the associated matrix of $\beta|_{\tilde{S}}$ is given by 
        \begin{equation*}
        \begin{bmatrix}
        \ast & 1 & 0 & 0\\
        1 & 0 & 0 & 0 \\
        0 & 0 & \ast & 1 \\
        0 & 0 & 1 & 0
        \end{bmatrix},
        \end{equation*}
where $\ast$ are suitable entries. Hence, the restriction of $\beta$ to $\tilde{S}$ is the sum $\beta_P(\tilde{y}) + \beta_P(\tilde{z})$ for suitable $\tilde{y}, \tilde{z} \in \KK$. Moreover, we can write $S = \tilde{S} + \tilde{S}^{\perp}$. By Lemma ~\ref{p_forms}, the restriction of $\beta$ to $\tilde{S}^{\perp}$ will be of the form $\beta_P(\tilde{a})$ for suitable $\tilde{a} \in \KK$. Thus, the sum of $\beta_P(y)$ with $\beta_{2P}(a)$ can be rewritten as the sum $\beta_P(\tilde{y}) + \beta_P(\tilde{z}) + \beta_P(\tilde{a})$. From here, the statement of the lemma follows.

\end{proof}
Lemma ~\ref{splitp} shows that any non-degenerate symmetric bilinear form on $nP$ is either the sum of $n/2$-copies of irreducible forms on $2P$ or the sum of $n$-copies of irreducible forms on $P$. We will show that in the former case, there are two distinct isomorphism classes that arise, whereas in the latter, there are infinitely many. We begin with the first case, which is easier to prove:

\begin{lem}\label{4Pbasis}
Suppose $\beta$ is a non-degenerate symmetric bilinear form on $nP$ such that 
\[\beta = \sum_{j=1}^{n/2} \beta_{2P}(a_j)\]
for $a_j \in \{0, 1\}$. Then, $\beta$ is in the same isomorphism class as one of the following two forms:
\begin{align*}
    \beta_{2P;0}^n &\coloneqq \frac{n}{2}\beta_{2P}(0) \\
    \beta_{2P;1}^n &\coloneqq \beta_{2P}(1) + \frac{n-2}{2}\beta_{2P}(0).
\end{align*}
The two forms are not isomorphic.
\end{lem}
\begin{proof}
We are done if for at most one value of $j$, we have $a_j = 1$. So let us suppose there are least two such values of $j$. Without loss of generality, we can assume they are the first two indices, i.\@e.\@ $a_1 = a_2 = 1$. Now, we will consider the direct summand $\beta_{2P}(a_1) + \beta_{2P}(a_2)$ of $\beta$, with basis $\{u_1, t.u_1, u_2, t.u_2\}$ for the first copy of $2P$ and $\{u_3, t.u_3, u_4, t.u_4\}$ a basis for the second copy of $2P$. We claim that this form can be written as $\beta_{2P}(0) + \beta_{2P}(1)$ by suitably changing basis. 

Let $u_5 = u_1 + u_3$, and let $u_6=u_2$. The associated matrix of $\beta$ restricted to the subobject $S_1$ spanned by $\{u_5, t.u_5, u_6, t.u_6\}$ (with respect to this basis) is given by \eqref{2p_form0}.
Similarly, define $u_7 = u_3 + t.u_2$ and $u_8 = u_2 + u_4$. The associated matrix of $\beta$ restricted to the subobject $S_2$ spanned by $\{u_7, t.u_7, u_8, t.u_8\}$ (with respect to this basis) is given by \eqref{2p_form1}. Moreover, we can see that $S_1$ and $S_2$ are orthogonal complements. This shows that $\beta_{2P}(1) + \beta_{2P}(1) = \beta_{2P}(0) + \beta_{2P}(1)$; the claim follows by induction. The two forms are not isomorphic because the form $\beta^n_{2P;0}$ is super-alternating, whereas the form $\beta_{2P;1}^n$ is not.
\end{proof}
We now consider the second case, where the non-degenerate symmetric bilinear form is the sum of forms on the object $P$. The procedure for doing so is more complicated than that of the first case. To start, we have the following lemma.

\begin{lem}\label{P+Pbasis}
For any $y\neq z \in \KK$, the form $\beta=\beta_P(y) + \beta_P(z)$ is in the same isomorphism class as $\beta_P(a) + \beta_P(y + z + a)$ for all $a \in \KK$.
\end{lem}
\begin{proof}
Let $\{u_1, t.u_1\}$ be a basis of the first $P$ object such that the associated matrix of $\beta_P(y)$ is given by \eqref{2p_form0}, and let $\{u_2, t.u_2\}$ be a basis of the second $P$ object such that the associated matrix of $\beta_P(z)$ is given by \eqref{2p_form1}. For some arbitrary $a \in \KK$, let $k=\sqrt{\frac{z+a}{z+y}}$, which is well-defined because $y\neq z$. Define $c = ky$ and $d = (1+k)x$. Then, $k(1+k)y + (1+k)kz + c(1+k) + dk = k((1+k)y+d) + (1+a)(az+c) = 0$. Now, let $u_3=ku_1+(1+k)u_2+ct.u_1+dt.u_2$, and let $u_4=(1+k)u_1+ku_2$. We have

\begin{itemize}
    \item $\beta(u_3,u_3)=k^2y+(1+k)^2z=k^2(y+z)+z=a$,
    \item $\beta(u_3,t.u_3)=k^2+(1+k)^2=1$,
    \item $\beta(u_3,u_4)=k(1+k)y+(1+k)kz+c(1+k)+dk=0$,
    \item $\beta(u_3,t.u_4)=k(1+k)+(1+k)k=0$,
    \item $\beta(u_4,u_4)=(1+k)^2y+k^2z=k^2(y+z)+y=y+z+a$, and
    \item $\beta(u_4,t.u_4)=(1+k)^2+k^2=1$.
\end{itemize}
Therefore, with respect to the basis $\{u_3, t.u_3, u_4, t.u_4\}$, the associated matrix of $\beta$ is
\begin{equation}\notag
		\bbordermatrix{& u_3& t.u_3& u_4& t.u_4\cr
		  &a & 1 &&\cr
		  &1 & 0 &&\cr
    	& & & y+z+a & 1 \cr
		  & & & 1 & 0 \cr
		}.
\end{equation}
This proves the claim.
\end{proof}

Now, our strategy will be to repeatedly use Lemma ~\ref{P+Pbasis} to convert a form that is the direct sum of forms described in Lemma ~\ref{p_forms} into a canonical form. For simplicity, we will refer to the process of identifying $\beta_P(y) + \beta_P(z)$ with $\beta_P(a) + \beta_P(y+z+a)$ as ``replacing $y, z$ by $a, y+z+a$". Given a form $\beta_P(y)$, we will refer to $y$ as the \textit{assigned scalar} of $\beta_P(y)$.

\begin{lem}\label{Pforms}
    Let $\beta$ be a non-degenerate symmetric bilinear form on the object $nP$ with $n > 1$ such that 
    \[\beta = \sum_{i=1}^n \beta_P(y_i)\]
    for suitable $y_i \in \KK$. If not all values of $y_i$ are the same, then we can write
    \[\beta = \beta_P(k) + \beta_P(1) + (n-2)\beta_P(0)\]
    for some suitable $k \in \KK$. If $n = 2$, then $k \neq 1$.
\end{lem}
\begin{proof}

First of all, let us suppose that $n = 2$. Then, we have $\beta = \beta_P(y_1) + \beta_P(y_2)$. We can replace $y_1, y_2$ with $1, y_1 + y_2 + 1$ and let $k = y_1 + y_2 + 1 \neq 1$. 

Now, suppose that $n \geq 3$. If $n-1$ of the assigned scalars are zero and the remaining scalar is $1$, then we are done. If instead the remaining scalar is some $\lambda \neq 0 \in \KK$, then we can do the replacement $\lambda, 0 \mapsto 1, \lambda + 1$, and we are done again. If $n-2$ of the assigned scalars are zero and the remaining two are $\lambda, \mu \in \KK-\{0\}$, then we can do the replacement $\lambda, \mu \mapsto 1, \lambda + \mu + 1$ if $\lambda \neq \mu$. If $\lambda = \mu$, we can first do the replacement $0, \lambda \mapsto 1, \lambda + 1$, then do the replacement $\lambda + 1, \mu \mapsto 1, 0$ (converting the three assigned scalars $\lambda, \mu, 0$ into $1, 1, 0$). This covers the case where $n-2$ assigned scalars are zero.
\par
Therefore, let us assume that at most $n-3$ of the assigned scalars are zero. If no assigned scalars are zero, we can find $y_a$ and $y_b$ with $y_a \neq y_b$ and do the replacement $y_a, y_b \mapsto 0, y_a+y_b$. Hence, we can ensure that least one of the assigned scalars is zero. If $n=3$, this returns us to the case where $n-2$ assigned scalars are zero. When $n>3$, we can find three additional assigned scalars $y_a$, $y_b$, and $y_c$ with $y_a \neq 0$. We can then perform the following iterative procedure until we arrive at a form that has $n-2$ zeroes as assigned scalars. Let $d$ be a nonzero scalar satisfying $d \neq y_b$ and $d \neq y_a + y_c$. We can do the replacements
\[ 0,y_a,y_b,y_c \mapsto d, y_a + d, y_b, y_c \mapsto 0, y_a + d, y_b + d, y_c \mapsto 0, 0, y_b+d, y_c + y_a + d,\]
where the notation is extended with two assigned scalars replaced in each step. These replacements give us an additional zero as an assigned scalar. The above process can be repeated until we have $n-2$ zeroes as assigned scalars, which is a case we have already considered. This proves the lemma.
\end{proof}

We combine our previous work to get the following theorem.
\begin{thm}\label{wxforms}
Any non-degenerate symmetric bilinear form $\beta$ on $nP$ lies in the isomorphism class of one of the following types of forms:
\begin{align*}
    \beta_{2P;0}^n &= \frac{n}{2}\beta_{2P}(0) \ \  (2 \mid n) \\
    \beta_{2P;1}^n &= \beta_{2P}(1) + \frac{n-2}{2} \beta_{2P}(0) \ \  (2 \mid n; n>0)\\ 
    \beta_{y,1; 0} &\coloneqq \beta_P(y) + \beta_P(1) \ \  (y \neq 1 \in \KK; n = 2) \\
    \beta_{y,1;n-2} &\coloneqq \beta_P(y) +\beta_P(1) + (n-2)\beta_P(0) \ \  (y \in \KK; n \geq 3) \\
    \beta_y^n &\coloneqq n\beta_P(y)\ \  (y \in \KK; n > 0).
\end{align*}
These forms are pairwise non-isomorphic, except some of the $\beta_{y,1;n-2}$ may represent the same isomorphism class for different $y$ (which we will see is not the case in Lemma \ref{fnon-isomorphic}).
\end{thm}
\begin{proof}
This follows from Lemma ~\ref{splitp}, Lemma ~\ref{4Pbasis}, and Lemma ~\ref{Pforms}. To see that we have distinct isomorphism classes, we will observe some properties about the forms. The first form is oscillating and super-alternating. The second form is not super-alternating but is oscillating. The remaining forms are not oscillating. Notice that $y\beta_y^n(u, t.u) = \beta^n_y(u, u)$ for all $u \in nP$, whereas for no $y \in \KK$ does there exist $z \in \KK$ such that $z\beta_{y,1;n-2}(u, u) = \beta_{y,1;n-2}(u, t.u)$ for all $u \in nP$. Therefore, we deduce that the $\beta_y^n$ are pairwise non-isomorphic and not isomorphic to anything else on the list. This proves the claim.
\end{proof}

\subsection{Classifying non-degenerate bilinear forms in the general case}
We now have classifications for the non-degenerate symmetric bilinear forms on objects of the form $m\un$ (Theorem ~\ref{vforms}) and for those on objects of the form $nP$ (Theorem ~\ref{wxforms}). In this section, we will use these results to provide the classification for any object $U \in \Ver_4^+$ with decomposition $U = m \un \oplus nP = V \oplus W \oplus X$ arising from the basis given by ~\eqref{standard_basis}.

\begin{lem}\label{1D}
Let $\beta$ be a non-degenerate symmetric bilinear form on $U \in \Ver_4^+$, and suppose that $U = V \oplus V^\perp$, where $V \cong m\un$, $V^\perp \cong nP$, and $\beta|_V = \alpha_1^m$. Then, either $\beta = \alpha^m_1 + n\beta_{2P,0}^n$ or $\beta = \alpha^m_1 + n\beta_{0}^n$.
\end{lem}
\begin{proof}
By Lemma ~\ref{splitp}, we know that $\beta$ is either in the same isomorphism class as
\begin{align*}
    \alpha_1^m + \sum_{i=1}^n \beta_P(y_i)
\end{align*}
or
\begin{align*}
    \alpha_1^m + \sum_{j=1}^{n/2} \beta_{2P}(a_j).
\end{align*}
Let us deal with the former case first. We claim that 
\[\alpha_1^1 + \beta_P(y_i) = \alpha_1^1 + \beta_P(0) \]
for all values of $y_i$. The associated matrix of the left-hand side is given by
\begin{equation*}
\bbordermatrix{& u_1& u_2 & t.u_2\cr
             & 1& &  \cr
             &  & y_i &1 \cr
             & & 1&0\cr}
\end{equation*}
in some suitable basis $\{u_1, u_2, t.u_2\}$. Let $u_3=u_1+\sqrt{y_i}t.u_2$ and $u_4=\sqrt{y_i}u_1+u_2$. Then, we can see that
\begin{itemize}
    \item $\beta(u_3,u_3)=1$,
    \item $\beta(u_3,u_4)=\sqrt{y_i}+\sqrt{y_i}=0$,
    \item $\beta(u_3,t.u_4)=0$,
    \item $\beta(u_4,u_4)=y_i+y_i=0$,
    \item $\beta(u_4,t.u_4)=1$, and
    \item the space spanned by $u_3$ is perpendicular to the space spanned by $\{u_4,t.u_4\}$.
\end{itemize}
In the basis $\{u_3, u_4, t.u_4\}$, the associated matrix is given by
\begin{equation*}
\bbordermatrix{&u_3 & u_4 &t.u_4 \cr
             & 1& &  \cr
             &  & 0&1 \cr
             & & 1&0\cr},
\end{equation*}
which shows the claim. Since $m > 0$, after iterating this procedure for each $i$, we see that 
\[\beta = \alpha_1^m + \sum_{i=1}^n \beta_P(y_i) = \alpha_1^m + n\beta_{0}.\]
Now, let us move to the second case, where 
\begin{align*}
    \beta = \alpha_1^m + \sum_{j=1}^{n/2} \beta_{2P}(a_j).
\end{align*}
We want to show that $\beta = \alpha_1^m + n\beta_{2P,0}$; this will follow if we can show that 

\[\alpha_1^1 + \beta_{2P}(1) = \alpha_1^1 + \beta_{2P}(0).\]
In other words, we need to find a change of basis so that we can go from the first matrix below to the second matrix below:
\begin{equation} \notag\bbordermatrix{&u_1&u_2&t.u_2&u_3&t.u_3\cr
		  &1&&&& \cr
       & & 1 & 0 & 0 & 1\cr  	
		  & & 0 & 0 & 1 & 0\cr
		  & & 0 & 1 & 0 & 0\cr  	
		& & 1 & 0 & 0 & 0\cr
		}
\rightarrow
\bbordermatrix{&u_4&u_5&t.u_5&u_6&t.u_6\cr
		  &1&&&& \cr
       & & 0 & 0 & 0 & 1\cr  	
		  & & 0 & 0 & 1 & 0\cr
		  & & 0 & 1 & 0 & 0\cr  	
		& & 1 & 0 & 0 & 0\cr
		}.
\end{equation}
Such a basis change is given by letting $u_4 = u_1 + t.u_3$, $u_5 = u_1 + u_2$, and $u_6 = u_3$. Iterating this for each value of $j$ such that $a_j = 1$ proves the second case.
\end{proof}
Using the previous lemma and our classifications on $V \cong m\un$ and $V^\perp \cong nP$, we obtain a classification of the non-degenerate symmetric bilinear forms on an object $U \cong m\un\oplus nP$. 

\begin{thm} \label{6forms}
Let $U\cong m\un\oplus nP$. Any non-degenerate symmetric bilinear form $\beta$ on $U$ is equivalent to one of the forms below:

\begin{equation}\label{eqn:A}\tag{A}
        \alpha_1^m + \frac{n}{2} \beta_{2P}(0)
 \ \  (m>0;2\mid n)
\end{equation} 

\begin{equation}\label{eqn:B}\tag{B}
        \alpha_1^m + n \beta_{P}(0)
 \ \  (m>0;n>0)
\end{equation}

\begin{equation}\label{eqn:C}\tag{C}
    \alpha_2^m + \frac{n}{2} \beta_{2P}(0)
 \ \  (2\mid m;2\mid n)
\end{equation} 

\begin{equation}\label{eqn:D}\tag{D}
    \alpha_2^m + \beta_{2P}(1) + \frac{n-2}{2} \beta_{2P}(0)
 \ \  (2\mid m;2\mid n;n\ge 2)
\end{equation} 

\begin{equation}\label{eqn:E}\tag{E}
        \alpha_2^m + n \beta_{P}(y)
 \ \  (y \in \KK; 2\mid m;n>0)
\end{equation} 

\begin{equation}\label{eqn:F}\tag{F}
    \alpha_2^m + (n-2)\beta_{P}(0) + \beta_P(1) + \beta_P(y)
 \ \ \left(\begin{aligned}&y \in \KK;\\&2\mid m;n\ge 2; \\ &(1+y,n)\neq (0,2) \end{aligned}\right)
\end{equation}
\end{thm}
\begin{proof}
Write $U = V \oplus V^\perp$ for some $V \cong  m\un$. If the restriction of $\beta$ to $V$ decomposes as $\alpha_1^m$, then Lemma ~\ref{1D} shows that $\beta$ is either in the isomorphism class ~\ref{eqn:A} or the isomorphism class ~\ref{eqn:B}. Otherwise, Theorem ~\ref{wxforms} gives a form belonging to one of the isomorphism classes ~\ref{eqn:C} through F. Since all alternating bilinear forms are symmetric, we have also classified all non-degenerate alternating bilinear forms on objects in $\ver_4^+$ (we will specify which forms are alternating in Theorem ~\ref{non-isomorphic}). In the next subsection, we will prove that forms in these isomorphism classes are pairwise non-isomorphic.
\end{proof}

\subsection{Proving Non-Isomorphism}
We start by describing basis-invariant properties of non-degenerate symmetric bilinear forms on objects $U \in \ver_4^+$. 

\begin{Def}
    Given a symmetric bilinear form $\beta$ on $U$, we define a \textit{good pair} as an ordered pair of scalars $(k, l) \in \KK^2$ satisfying $k\beta(u, t.u) = l\beta(u, u)$ for all $u \in U$.
\end{Def}

\begin{prop}\label{property}
    Let $\beta$ be a symmetric bilinear form on $U$, and let $k, \ell$ be scalars in $\KK$. The set of vectors $u \in U$ that satisfy $k\beta(u, t.u) = l\beta(u, u)$ is a subspace of $U$.
\end{prop}

\begin{proof}
If $u_1,u_2 \in U$ satisfy this equation, then so does their sum:
\begin{align*}
  k\beta(u_1+u_2,t.(u_1+u_2))&=k\beta(u_1+u_2,t.u_1+t.u_2) \\
  &=k\beta(u_1,t.u_1)+k\beta(u_2,t.u_2)+k\beta(u_1,t.u_2)+k\beta(u_2,t.u_1) \\
  &=\ell\beta(u_1,u_1)+\ell\beta(u_2,u_2)+2k\beta(u_1,t.u_2) \\
  &=\ell\beta(u_1,u_1)+\ell\beta(u_2,u_2)+2\ell\beta(u_1,u_2) \\
  &=\ell\beta(u_1+u_2,u_1+u_2).
\end{align*}
Moreover, for any scalar $\lambda$, 
\begin{align*}
k\beta(\lambda u_1,t.(\lambda u_1))=k\beta(\lambda u_1,\lambda t.u_1)=k\lambda ^2\beta(u_1,t.u_1)=\ell \lambda^2\beta(u_1,u_1)=\ell \beta(\lambda u_1,\lambda u_1).
\end{align*}
\end{proof}
In the case that a symmetric bilinear form has the good pair $(1, 0)$, we recover the definition of an oscillating bilinear form. In the case that a symmetric bilinear form has the good pair $(0, 1)$, we recover the definition of a super-alternating bilinear form. Alternating forms in our classification have additional invariant properties:

\begin{lem}\label{fdefined}
Let $\beta$ be a non-degenerate alternating bilinear form on $U$. Suppose $u_1, u_2 \in U$ such that $u_1 - u_2 \in \ker t$. Then, $\beta(u_1,u_1) = \beta(u_2,u_2)$.
\end{lem}
\begin{proof}
By Proposition ~\ref{alternatingmaindef}, the fact that $u_1-u_2 \in \ker t$ implies $\beta(u_1-u_2, u_1-u_2)=0$. We have 
\begin{align*}
    0 &= \beta(u_1 - u_2, u_1 - u_2) \\
    &= \beta(u_1, u_1) - \beta(u_1, u_2) - \beta(u_2, u_1) + \beta(u_2, u_2)  \\
    &= \beta(u_1, u_1) + \beta(u_2, u_2),
\end{align*}
which means $\beta(u_1, u_1) = \beta(u_2, u_2)$, which shows the claim.
\end{proof}
This motivates the following definition:

\begin{Def}
    Let $\beta$ be a non-degenerate alternating bilinear form on $U$. The \emph{$X$-function} $f_\beta: X \to \KK$ of $\beta$ is defined by $f_\beta(x)=\beta(u,u)$, where $x \in X$ and $u \in U$ is in the preimage of $x$ under the map of the $t$-action.
\end{Def}
Observe that by Proposition \ref{bilform_calcs} $\beta|_{X \otimes \ker t}$ is the zero-map. This motivates another definition. 

\begin{Def}
    The restriction $\beta|_{X \otimes U}$ factors through a map $X \otimes U/\ker t \rightarrow \un$, which can be identified with a map $g_\beta: X \otimes X \rightarrow \un$. We call $g_\beta$ the \emph{$X$-form}  of $\beta$. Explicitly, this map is given by $g_\beta(x_1, x_2) \coloneqq \beta(x_1, u_2)$, where $u_2 \in U$ is any preimage of $x_2$ under the $t$-action.    
\end{Def}

\begin{prop}\label{gnondegen}
Let $\beta$ be a non-degenerate alternating bilinear form on $U$. The $X$-form $g_\beta$ of $\beta$ is non-degenerate, symmetric, and bilinear.
\end{prop}
\begin{proof}
First, suppose for the sake of contradiction that $g_\beta$ is degenerate. Then, there exists a vector $x\in X$ such that $g(x,x')=0$ for all $x' \in X$. Thus, for any vector $u'$ such that $t.u'\in X$, $\beta(x,u')=0$. However, $X$ is the image of $U$ under the $t$-action, so $\beta(x,u')=0$ for all $u' \in U$, which is impossible because $\beta$ is non-degenerate. The map is obviously bilinear and easily checked to be symmetric.
\end{proof}

\begin{Def}
    Let $\beta$ be a non-degenerate alternating bilinear form on $U$. Given a basis of $X$, the \emph{$X$-matrix} of $\beta$ is the associated matrix of the $X$-form of $\beta$.
\end{Def}

Because the $X$-form is non-degenerate for any non-degenerate alternating bilinear form, we know that the $X$-matrix is always invertible. Next, we introduce the basis-invariant notion of the \textit{form invariant} to distinguish between isomorphism classes of forms.

\begin{Def}\label{invariantdef}
Suppose that $\beta$ is a non-degenerate alternating bilinear form on $U$. Let $\{\chi_1, \dots, \chi_n\}$ be a basis of $X$, and denote the $X$-matrix of $\beta$ with respect to this basis by $M$. The \textit{form invariant} of $\mathcal{I}_\beta$ of $\beta$ is the sum $\sum_{i=1}^n f_\beta(\chi_i)(M^{-1})_{ii}$.
\end{Def}

\begin{rem}\label{invariant_nP}
    Let $\eta$ be a non-degenerate alternating bilinear form on an object $R \in \Ver_4^+$ with decomposition $R = p\un \oplus qP$. The formula for $\mathcal{I}_\eta$ depends only on the restriction of $\eta$ to $qP$, so $\mathcal{I}_{\eta}=\mathcal{I}_{\eta|_{qP}}$. 
\end{rem}

\begin{thm}\label{invariant}
Let $\beta$ be a non-degenerate alternating bilinear form on $U$. The form invariant of $\beta$ is basis-invariant.
\end{thm}
\begin{proof}
Denote the $X$-function and $X$-form of $\beta$ by $f$ and $g$, respectively, and with respect to the basis $\{x_1, x_2, \dots, x_m\}$ of $X$, define $M$ to be the $X$-matrix of $\beta$. Given an invertible linear transformation $A: X \to X$, we want to show that when evaluated on the basis $\{Ax_1,Ax_2,\dots ,Ax_n\}$, the form invariant remains unchanged. First, we show that the associated matrix of $g$ with respect to this basis is $A^\top MA$. Using the property that $g$ is bilinear, we can rewrite each entry of this associated matrix as follows: $$g(Ax_i,Ax_j)=\underset{1\le k,\ell\le n}{\sum}A_{ki}A_{\ell j}g(x_k,x_\ell)=\underset{1\le k,\ell\le n}{\sum}A_{ki}A_{\ell j}M_{k\ell}=\underset{1\le k,\ell\le n}{\sum}A^\top_{ik}M_{k\ell}A_{\ell j}=(A^\top MA)_{ij}.$$ 
Additionally, we have $$f(Ax_i)=\beta\left(\sum_{j=1}^nA_{ji}w_j,\sum_{k=1}^nA_{ki}w_k\right)=\sum_{j=1}^n\sum_{k=1}^nA_{ji}A_{ki}\beta(w_j,w_k).$$ For each pair $(a, b)$ where $1 \leq a,b \leq n$, we have $A_{ai}A_{bi}\beta(w_a,w_b)=A_{bi}A_{ai}\beta(w_b,w_a)$, which implies $A_{ai}A_{bi}\beta(w_a,w_b)+A_{bi}A_{ai}\beta(w_b,w_a)=0$ in characteristic $2$. Therefore, we can simplify $f(Ax_i)$ to $$\sum_{j=1}^nA_{ji}^2\beta(w_j,w_j)=\sum_{j=1}^nA_{ji}^2f(x_j).$$

We want to prove $$\sum_{i=1}^nf(x_i)(M^{-1})_{ii}=\sum_{i=1}^n\sum_{j=1}^nA_{ji}^2f(x_j)(A^\top MA)^{-1}_{ii},$$ and it suffices to show that $$(M^{-1})_{ii}=\sum_{k=1}^nA_{ik}^2(A^\top MA)_{kk}^{-1}.$$ The matrix $M^{-1}$ can be written as $A(A^\top MA)^{-1}A^\top$. Thus, 
\begin{equation*}
    M^{-1}_{ii} = \sum_{1\le j,k\le n}A_{ij} (A^\top MA)^{-1}_{jk} A^\top_{ki}=\sum_{1\le j,k\le n}A_{ij}A_{ik}(A^\top MA)^{-1}_{jk}.
\end{equation*} Since $A^\top MA$ is symmetric, $(A^\top MA)^{-1}$ must also be symmetric. 

For each pair $(a,b)$ where $1 \leq a,b\leq n$, we have $A_{ia}A_{ib}(A^\top MA)^{-1}_{ab}=A_{ib}A_{ia}(A^\top MA)^{-1}_{ba}$, which means that $A_{ia}A_{ib}(A^\top MA)^{-1}_{ab}+A_{ib}A_{ia}(A^\top MA)^{-1}_{ba}=0$. Then, $$\sum_{1\le j,k\le n}A_{ij}A_{ik}(A^\top MA)^{-1}_{jk}=\underset{1 \le k \le n}{\sum}A_{ik}A_{ik}(A^\top MA)^{-1}_{kk}=\sum_{k=1}^nA_{ik}^2(A^\top MA)_{kk}^{-1},$$ as desired.
\end{proof}

We are now ready to prove non-isomorphism.

\begin{lem}\label{fnon-isomorphic}
For all $a, b \in \KK$, forms in the class \ref{eqn:F}($1+a$) and forms in the isomorphism class \ref{eqn:F}($1+b$) are isomorphic only if $a=b$.
\end{lem}
Let $\beta$ be a form in F($1+a$). We will use the basis given by ~\eqref{standard_basis} to represent the associated matrix of $\beta$ in Theorem \ref{6forms}. With respect to the basis $\{x_1, x_2, \dots, x_n\}$, the $X$-matrix $M$ of $\beta$ is the identity matrix $I_n$. Then, $\sum_{i=1}^nf(x_i)(M^{-1})_{ii}=\sum_{i=1}^nf(x_i)$, which is the sum of the diagonal entries of $M$. The form invariant of $\beta$ thus evaluates to $\mathcal{I}_\beta=1+a$. Since $1+a = 1+b$ only if $a=b$, this proves the lemma.

\begin{thm}\label{non-isomorphic}
The forms described in Theorem ~\ref{6forms} are pairwise non-isomorphic. 
\end{thm}
\begin{proof}
By Proposition ~\ref{alternatingfirstdef}, the alternating bilinear forms in our classification are those that vanish on $v_j \otimes v_j$ for $1 \leq j \leq m$. We deduce that forms in the isomorphism classes ~\ref{eqn:A} and ~\ref{eqn:B} are not alternating, while forms of the remaining four classes are. Thus, forms in ~\ref{eqn:A} and ~\ref{eqn:B} are not isomorphic to forms in the other classes.

By Proposition ~\ref{property}, we can determine the good pairs of forms in our classification by examining the properties of vectors in a basis of $U$. Forms belonging to ~\ref{eqn:B} and \ref{eqn:F} have a single good pair $(0, 0)$, whereas the good pairs of forms in ~\ref{eqn:A} and ~\ref{eqn:D} are $(k, 0)$ for all scalars $k$, Forms in E($a$) where $a \in \KK$ have the good pairs $(ka, k)$ for all scalars $k$. For all $k, \ell \in \KK, u \in U$, $k\beta(u, t.u) = \ell\beta(u, u) = 0$ for all forms $\beta$ in ~\ref{eqn:C}. Therefore, forms in \ref{eqn:C} have the good pair $(k, \ell)$ for all scalars $k, \ell$.

We can use the criterion of distinct good pairs to conclude that forms in ~\ref{eqn:A} and ~\ref{eqn:B} are not isomorphic and forms belonging to the classes ~\ref{eqn:C}, ~\ref{eqn:D}, \ref{eqn:E}, and \ref{eqn:F} are pairwise non-isomorphic. Finally, we proved in Lemma ~\ref{fnon-isomorphic} that the forms in F$(1+a)$ and forms in F$(1+b)$ with $a \neq b \in \KK$ are distinct.
\end{proof}

We finish this section with calculating the form invariants of the forms described by ~\ref{eqn:C}, ~\ref{eqn:D}, and \ref{eqn:E}. This information becomes useful in the next section, where we determine the sums and products of bilinear forms described by our isomorphism classes.
\begin{prop}\label{invariantlist}
The form invariants of forms in ~\ref{eqn:C} and ~\ref{eqn:D} are zero, and for $a \in \KK$, the form invariant of forms in \ref{eqn:E}($a$) is $na$.
\end{prop}
\begin{proof}
Suppose $\beta$ is a non-degenerate symmetric bilinear form in E($a$). Again, we use the basis given by ~\eqref{standard_basis} to represent the associated matrix of $\beta$ in Theorem \ref{6forms}. The $X$-matrix of $\beta$ with respect to this basis is the identity matrix $I_n$, and for $1\le i\le n$,  $f_\beta(x_i)=a$. The form invariant of $\beta$ evaluates to $\mathcal{I}_\beta=na$.

Now, suppose $\beta$ is a form in ~\ref{eqn:C} or ~\ref{eqn:D}.  With respect to the same basis, the $X$-matrix of $\beta$, which we will once again denote $M$, is direct sums of the $2\times 2$ matrix given by \[\begin{bmatrix} 0&1\\1&0\end{bmatrix}.\] Since $M$ is its own inverse, $M^{-1}_{ii}=0$ for $1\le i\le n$. Thus, $\mathcal{I}_\beta=0$.
\end{proof}

\subsection{Classification of Non-Degenerate Quadratic Forms}
Having classified the non-degenerate symmetric forms, we can now classify the isomorphism classes non-degenerate quadratic forms in $\Ver_4^+$. As with symmetric forms, when we say two quadratic forms are equal, we always mean up to equivalence.
\par
First, let us recall the classification of non-degenerate quadratic forms on arbitrary vector spaces over $\KK$. If $T$ is a two-dimensional vector space spanned by $\{v, w\}$, then we have a single non-degenerate quadratic form up to isomorphism, and it is given by $q(av + bw) = ab$. Let us denote this form by $\mathbb{H}$ (as it is known as a \textit{hyperbolic plane}). The following theorem is well-known (see  in \cite[Chapter 2]{elman2008algebraic}):

\begin{thm}\label{quadratic_form_ordinary_classification}
Let $q$ be a non-degenerate quadratic form on a vector space $V$. Then, $\dim V$ is even and $q = \tfrac{\dim V}{2}\mathbb{H}$.
\end{thm}
Now, we are ready to turn to the classification of non-degenerate quadratic forms in $\Ver_4^+$.  

\begin{lem}
Let $q$ be a non-degenerate quadratic form on $U \in \Ver_4^+$ such that $U = m\un \oplus nP$. Then $m$ is necessarily even and $q$ is equivalent to the quadratic form $\tfrac{m}{2}\mathbb{H} + q_\gamma$, where $\gamma$ is a bilinear form on $nP$ given by those in Theorem \ref{wxforms} and where $q_\gamma$ is the quadratic form associated to $\gamma$ (recalling Corollary \ref{frobtwist0}).
\end{lem}

\begin{proof}
Because the radical of $\beta_q$ is non-degenerate, we can choose a decomposition $U = V \oplus V^\perp$ so that $\beta_q$ is one of the forms in Theorem \ref{6forms}, where $V \cong m\un$ and $\beta_q = \beta_q|_V + \beta_q|_{V^\perp}$. Then, it follows by Proposition \ref{restrict_beta_q} that $q = q|_V + q|_{V^\perp}$. Now, we can apply Theorem \ref{quadratic_form_ordinary_classification} and Theorem \ref{wxforms}, respectively, to $q|_V$ and $q|_{V^\perp}$ to get the desired result. The existence of an isomorphism would contradict Theorem \ref{6forms}, so we know they are pairwise non-isomorphic.
\end{proof}

\subsection{Further Directions}
Having classified the non-degenerate symmetric bilinear forms and non-degenerate quadratic forms in $\Ver_4^+$, there are some natural further directions to consider. Let us call a group scheme in $\Ver_4^+$ preserving a symmetric bilinear form a symplectic group, and the group scheme preserving a quadratic form an orthogonal group.
\par
The first question is to see what the right notion of a non-degenerate quadratic form is. By tweaking the definition of non-degeneracy for a quadratic form, we will get different orthogonal groups that stabilize them. Ideally, these groups would be simple or close to being simple. Can we use this as a guiding principle to find what a non-degenerate quadratic form should be?
\par
In the case that a form is just on the object $nP$, then the symplectic and orthogonal group coincide because $\wedge^2(nP) = S^2(nP) \cong \Gamma^2(nP)$. How do these groups vary if we change the form on $nP$? If we change parameters, do groups remain isomorphic? Similar questions can be asked for the corresponding Lie algebras.
\par
Finally, forms on $nP$ appear to be analogous to periplectic forms on spaces of superdimension $n|n$ in characteristic zero. Do the Lie algebras that preserve this form lift to the periplectic Lie superalgebra $\mathfrak{p}(n)$ in characteristic zero (or perhaps over the $2$-adic field $\mathbb{Z}_2[\sqrt{2}]$)? A simple first step would be to check if dimensions match.

\section{Witt Semi-Ring Structure}\label{witt}
In this section, we describe the structure of the Witt semi-ring of isomorphism classes of non-degenerate symmetric bilinear forms in $\Ver_4^+$ (see \S \ref{forms_operations}). Our results are provided in the table at the end of each subsection. As a set, the elements of the Witt semi-ring are the isomorphism classes of the non-degenerate symmetric bilinear forms described in Theorem $~\ref{6forms}$.
\par

Throughout this section, we let $\beta$ and $\eta$ denote non-degenerate symmetric bilinear forms on objects $U, R \in \ver_4^+$, respectively. We fix a basis of $U=m \un \oplus nP$ as given by \eqref{standard_basis}, and we fix a basis of $R=p \un\oplus qP$ by $$\{\nu_1,\nu_2, \dots, \nu_p, \omega_1, \chi_1, \dots, \omega_q,\chi_q\},$$ where $t.\nu_j=0$ for all $1\le j\le p$ and $t.\omega_k=\chi_k$ for all $1\le k\le q$. These bases are chosen so that $\beta$ and $\eta$ have associated matrices as described in Theorem \ref{6forms}. Given $\beta$ and $\eta$, we determine which isomorphism classes their sum and product belong to (denoted \ref{eqn:A} through \ref{eqn:F}, as labeled in Theorem ~\ref{6forms}). 

\subsection{Additive Structure}
In this section, we describe the invariant properties of $\beta + \eta$, which will enable us to classify the form up to isomorphism.

\begin{lem}
The good pairs of $\beta + \eta$ are the intersection of the good pairs of $\beta$ and the good pairs of $\eta$.
\end{lem}
\begin{proof}
    Let $k, \ell$ be scalars in $\KK$. If $k\beta(u, t.u) = \ell\beta(u, u)$ for all $u \in U$ and $k\eta(r, t.r) = \ell\eta(r,r)$ for all $r \in R$, 
    we have 
    \begin{align*}
        &k\beta(u, t.u) + k\eta(r, t.r) = \ell\beta(u, u) + \ell\eta(r, r) \\ \implies &k(\beta + \eta)(u + r, t.(u + r)) = \ell(\beta +\eta)(u + r,u + r).
    \end{align*}
    
    For the converse, we suppose $(k, \ell)$ is a good pair of $\beta + \eta$, meaning
    \begin{equation}\label{4.1}
        k(\beta +\eta)(u + r, t.(u + r)) = \ell(\beta +\eta)(u + r, u + r)
    \end{equation}
    for all $u, r \in U, R$. We have $\ell(\beta +\eta)(u + r, u + r) = \ell\beta(u, u) + \ell\eta(r,r)$, and the left-hand side of ~\eqref{4.1} evaluates to
    \begin{align*}
        k(\beta +\eta)(u + r, t.(u + r)) = k(\beta +\eta)(u + r, t.u + t.r) = k\beta(u, t.u) + k\eta(r, t.r).    
    \end{align*}
    Thus, we can rewrite ~\eqref{4.1} as
    \begin{equation*}
        k\beta(u, t.u) + k\eta(r, t.r) = \ell\beta(u, u) + \ell\eta(r, r).
    \end{equation*}
    Setting $r = 0$ in the equation above yields $k\beta(u, t.u) = \ell\beta(u, u)$, and setting $u = 0$ yields $k\eta(r, t.r) = \ell\eta(r, r)$.
\end{proof}

\begin{lem}\label{alternatingiff}
The sum $\beta + \eta$ is alternating if and only if both $\beta$ and $\eta$ are alternating.
\end{lem}
\begin{proof}
    Decompose $U=V_U\oplus W_U\oplus X_U$ and $R=V_R\oplus W_R\oplus X_R$. If $\beta$ and $\eta$ are alternating, then by Proposition ~\ref{alternatingfirstdef}, $\beta(a, a) = 0$ for all $a \in V_U \oplus X_U$, and $\eta(b, b) = 0$ for all $b \in V_R \oplus X_R$.
    Then, $(\beta + \eta)(a + b, a + b) = \beta(a, a) + \eta(b, b) = 0$ for all $a \in V_U \oplus X_U, b \in V_R \oplus X_R$, which proves by Proposition ~\ref{alternatingfirstdef} that $\beta + \eta$ is alternating.
    
    To prove the converse, we will show that $\beta + \eta$ is not alternating when at least one of $\beta$ and $\eta$ is not alternating. If $\beta$ is not alternating, then Proposition \ref{alternatingfirstdef} implies the existence of a vector $v_1\in V_U$ such that $\beta(v_1,v_1)\neq 0$. For any vector $\chi$ in $X_R$, $t.(v_1+\chi)=t.v_1+t.\chi=0$, and $\eta(\chi,\chi)=0$. Consequently, $(\beta +\eta)(v_1+\chi,v_1+\chi)=\beta(v_1,v_1)+\eta(\chi,\chi)\neq 0$, and it follows from Proposition ~\ref{alternatingmaindef} that $\beta +\eta$ is not alternating.
\end{proof}

\begin{lem}\label{invariantsum}
If both $\beta$ and $\eta$ are alternating, then $\mathcal{I}_{\beta + \eta}=\mathcal{I}_\beta+\mathcal{I}_\eta$.
\end{lem}
\begin{proof}
    First, let us establish our notation for this proof. The bases of $X_U$ and $X_R$ are given by $\{x_1,x_2,\dots ,x_n\}$ and $\{\chi_1,\chi_2,\dots ,\chi_q\}$, respectively. We denote the $X$-function of $\beta$ by $f_\beta$, the $X$-function of $\eta$ by $f_\eta$, and the $X$-function of $\beta + \eta$ by $f_{\beta + \eta}$. Additionally, $X$-matrices of $\beta$, $\eta$, and $\beta + \eta$ are denoted by $M_\beta$, $M_\eta$, and $M$, respectively. 
    
    Define a basis of $\beta + \eta$ by $\{b_1, \dots, b_{n+q}\}$ where the vectors $b_1, \dots, b_n$ are given by $x_1, \dots, x_n$ and the vectors $b_{n+1}, \dots, b_{n+q}$ are given by $\chi_1, \dots, \chi_q$. For any $1\le i\le n$, $f_{\beta + \eta}(x_i+0)=\beta(x_i,x_i)=f_{\beta}(x_i)$. We also have $f_{\beta + \eta}(0+\chi_i)=f_{\eta}(\chi_i)$ for all $1 \leq i \leq q$. There is a similar relationship between the $X$-matrices of our forms: $M = M_\beta \oplus M_\eta =
            \left[\begin{matrix}
    		  M_\beta & 0 \\
                0 & M_\eta \\
    		\end{matrix}\right]$,
    so
    $M^{-1} = \left[\begin{matrix}
    		  M_\beta^{-1} & 0 \\
                0 & M_\eta^{-1} \\
    		\end{matrix}\right]$.
    Thus,
    \begin{align*}
    \mathcal{I}_{\beta + \eta} &= \sum_{i=1}^{n+q} f_{\beta + \eta}(b_i)(M^{-1})_{ii} \\
    &=\sum_{i=1}^{n} f_{\beta + \eta}(x_i+0)(M^{-1})_{ii}+\sum_{i=n+1}^{n+q}f_{\beta + \eta}(0+\chi_{i-n})(M^{-1})_{ii} \\
    &\quad= \sum_{i=1}^{n} f_{\beta}(x_i)(M_\beta^{-1})_{ii} + \sum_{i=1}^{q} f_{\eta}(\chi_i)(M_\eta^{-1})_{ii}.
    \end{align*}
\end{proof}
We can now apply our work from the previous section on good pairs and alternating forms (Theorem ~\ref{non-isomorphic}) and form invariants (Lemma ~\ref{fnon-isomorphic}, Proposition ~\ref{invariantlist}) to determine the  sum of isomorphism classes in our Witt semi-ring in the table below. We list the isomorphism classes of $\beta$ (on $m\un \oplus nP$) and $\eta$ (on $p\un \oplus qP$) in the top row and the leftmost column, respectively.
\begin{center}
\begin{tabular}{|l|l|l|l|l|p{60mm}|l|}
\hline
$+$       & A & B & C & D & E($a$) & F($a$) \\ \hline
A      & A & B & A & A & B & B \\ \hline
B      &   & B & B & B & B & B \\ \hline
C      &   &   & C & D & E($a$) & F($a$) \\ \hline
D      &   &   &   & D & F($na$) & F($a$) \\ \hline
E($b$) &   &   &   &      & $a=b\rightarrow$ E$(a)$; \newline$a\neq b\rightarrow$ F($   na+qb$) & F($a+qb$)    \\ \hline
F($b$) &   &   &   &     &      &  F($a+b$)\\ \hline
\end{tabular}
\end{center}
In the table, $a$ and $b$ represent arbitrary scalars.  The blank entries are given by commutativity.

\subsection{Multiplicative Structure}

To determine the product on bilinear forms, we will employ a similar strategy as the one we used to find the sum.

\begin{rem}
    Some statements in this section assume properties for at least one of $\beta$ and $\eta$ or assume different properties for $\beta$ and $\eta$. By commutativity, these claims are also true when we interchange the assumptions for $\beta$ and the assumptions for $\eta$.
\end{rem}

First, we will determine the good pairs of $\beta\times\eta$. By Proposition \ref{property}, it suffices to consider the pairs $(k,\ell) \in \KK^2$ that satisfy the property $$k(\beta\times\eta)(b_1\otimes b_2,t.(b_1\otimes b_2))=\ell(\beta\times\eta)(b_1\otimes b_2,b_1\otimes b_2)$$ for all vectors $b_1 \otimes b_2$ in a basis of $U \otimes R$. It is easier for us to instead consider the pairs $(k,\ell) \in \KK$ that satisfy this property for vectors of the form $u \otimes r \in U \otimes R$. This will give us all of the good pairs of $\beta \times \eta$ because the set of all vectors in $U \otimes R$ expressible as $u \otimes r$ contains a basis for $U \otimes R$. For vectors of this form, we have
\begin{equation}\label{4.2}
    \begin{aligned}
        (\beta\times\eta)(u\otimes r,u\otimes r)&=\beta(u,u)\eta(r,r)+\beta(u,t.u)\eta(r,t.r), \\
        (\beta\times\eta)(u\otimes r,t.(u\otimes r))&=(\beta\times\eta)(u\otimes r,t.u\otimes r+u\otimes t.r)\\
        &=(\beta\times\eta)(u\otimes r,t.u\otimes r) + (\beta\times\eta)(u\otimes r,u\otimes t.r), \\
        &=\beta(u,t.u)\eta(r,r)+\beta(u,u)\eta(r,t.r).
    \end{aligned}
\end{equation}

We begin with the cases where at least one of $\beta$ and $\eta$ lies in the isomorphism classes ~\ref{eqn:C} or \ref{eqn:E}(1).

\begin{prop}
If $\beta$ lies in ~\ref{eqn:C}, then $\beta\times\eta$ must also belong to ~\ref{eqn:C}.
\end{prop}
\begin{proof}
Since $\beta$ is in \ref{eqn:C}, $\beta(u,t.u)=0$ and $\beta(u,u)=0$ for all $u \in U$. For all $r \in R$, we thus have $(\beta\times\eta)(u\otimes r,u\otimes r)=0$ and $(\beta\times\eta)(u\otimes r,t.(u\otimes r))=0$ by the equations in ~\eqref{4.2}. These properties are only exhibited by forms in ~\ref{eqn:C}.
\end{proof}
\begin{prop}
Suppose that $\eta$ lies in \ref{eqn:E}(1) and $\beta$ does not belong to the isomorphism classes ~\ref{eqn:C} or \ref{eqn:E}(1). Then, $\beta\times\eta$ is in \ref{eqn:E}(1).
\end{prop}
\begin{proof}
The equation $(\beta\times\eta)(u\otimes r,t.(u\otimes r))=(\beta\times\eta)(u\otimes r,u\otimes r)$ holds for all vectors of the form $u \otimes r$ in $U \otimes R$. We can see that $(1, 1)$ is a good pair of $\beta\times\eta$, which is only true for forms belonging to the classes $\ref{eqn:C}$ and \ref{eqn:E}(1). Since $\beta$ is not in $\ref{eqn:C}$ or \ref{eqn:E}(1), there exists a vector $u_1\in U$ such that $\beta(u_1,u_1)\neq \beta(u_1,t.u_1)$. Furthermore, since $\eta$ is in \ref{eqn:E}(1), there exists a vector $r_1\in \eta$ such that $\eta(r_1,r_1)=\eta(r_1,t.r_1)\neq 0$. Then, $(\beta\times\eta)(u_1\otimes r_1,u_1\otimes r_1)$ must be nonzero, which cannot be true for forms in ~\ref{eqn:C}.
\end{proof}
\begin{prop}
If $\beta$ and $\eta$ are both in \ref{eqn:E}(1), then $\beta\times\eta$ belongs to ~\ref{eqn:C}.
\end{prop}
\begin{proof}
If $\beta$ and $\eta$ are both in \ref{eqn:E}(1), then they must each have the good pair $(1, 1)$. In other words, $\beta(u, t.u)=\beta(u, u)$ for all $u \in U$, and $\eta(r, t.r) = \eta(r, r)$ for all $r\in R$.
For all values of $u \otimes r \in U \otimes R$, we thus have
\begin{align*}
    (\beta\times\eta)(u\otimes r,u\otimes r)&=\beta(u,u)\eta(r,r)+\beta(u,t.u)\eta(r,t.r)
    = 2 \cdot \beta(u,u)\eta(r,r)
    = 0, \\
    (\beta\times\eta)(u\otimes r,t.(u\otimes r))&=\beta(u,t.u)\eta(r,r)+\beta(u,u)\eta(r,t.r)= 2 \cdot \beta(u,u)\eta(r,r)=0. 
\end{align*}
These equations only hold for forms in ~\ref{eqn:C}.
\end{proof}

The remaining cases occur when neither $\beta$ nor $\eta$ belongs to ~\ref{eqn:C} or \ref{eqn:E}(1). To address these cases, we start with the following proposition.
\begin{prop}\label{system}
    Suppose $\beta$ has a single good pair $(0, 0)$. For any scalars $k, \ell$, there exists a solution to the system of equations 
    \begin{align*}
        &\beta(u, u) = k, \\
        &\beta(u, t.u) = \ell.
    \end{align*}
\end{prop}
\begin{proof}
Since $(0,0)$ is the only good pair of $\beta$, there exists a vector $\mu_1\in U$ such that at least one of $\beta(\mu_1,\mu_1)$ and $\beta(\mu_1,t.\mu_1)$ is nonzero. Let $k_1, \ell_1$ be the scalars given by $k_1 \coloneqq \beta(\mu_1, \mu_1)$ and $\ell_1 \coloneqq \beta(\mu_1, t.\mu_1)$. Then, $(k_1, \ell_1) \neq (0, 0)$.
If $\beta(\mu_1, \mu_1)\beta(\mu, t.\mu)=\beta(\mu_1, t.\mu_1)\beta(\mu, \mu)$ for all $\mu \in U$, then $(k_1, \ell_1)$ would be a good pair of $\beta$.
Therefore, since $(0, 0)$ is the only good pair of $\beta$, there must exist some vector $\mu_2 \in U$ such that $$k_1\beta(\mu_2, t.\mu_2) \neq \ell_1\beta(\mu_2, \mu_2).$$
Defining $k_2 \coloneqq \beta(\mu_2,\mu_2)$ and $ \ell_2 \coloneqq \beta(\mu_2,t.\mu_2),$ we have $k_1\ell_2\neq k_2\ell_1$. The pairs $(k_1,\ell_1),(k_2,\ell_2)$ are linearly independent vectors over $\KK^2$, so $(k_1,\ell_1),(k_2,\ell_2)$ span $\KK^2$. Thus, there exist scalars $c, d$ such that $c(k_1, \ell_1) + d(k_2, \ell_2) = (k, \ell)$. 

Let $u = \sqrt{c}\mu_1 + \sqrt{d}\mu_2$. We have
    \begin{align*}
        \beta(u, u) &= \beta(\sqrt{c}\mu_1 + \sqrt{d}\mu_2, \sqrt{c}\mu_1 + \sqrt{d}\mu_2) \\
        &= c\beta(\mu_1, \mu_1) + d\beta(\mu_2, \mu_2) + 2\cdot \sqrt{cd}\beta(\mu_1, \mu_2)) \\
        &= c\beta(\mu_1, \mu_1) + d\beta(\mu_2, \mu_2)
        = k
    \end{align*}
    and
    \begin{align*}
        \beta(u, t.u) &= \beta(\sqrt{c}\mu_1 + \sqrt{d}\mu_2, t.(\sqrt{c}\mu_1 + \sqrt{d}\mu_2)) \\
        &= \beta(\sqrt{c}\mu_1 + \sqrt{d}\mu_2, \sqrt{c}t.\mu_1 + \sqrt{d}t.\mu_2) \\
        &= c\beta(\mu_1, t.\mu_1) + d\beta(\mu_2, t.\mu_2) + \sqrt{cd}(\beta(\mu_1, t.\mu_2) + \beta(t.\mu_1, \mu_2)) \\
        &= c\beta(\mu_1, t.\mu_1) + d\beta(\mu_2, t.\mu_2) + \sqrt{cd}(\beta(\mu_1, t.\mu_2) + \beta(\mu_1, t.\mu_2)) \\
        &= c\beta(\mu_1, t.\mu_1) + d\beta(\mu_2, t.\mu_2) + 2 \cdot \sqrt{cd}\beta(\mu_1, t.\mu_2) \\
        &= c\beta(\mu_1, t.\mu_1) + d\beta(\mu_2, t.\mu_2)
        = \ell,
    \end{align*}
    which shows that $u$ is a solution to the system.
\end{proof}

\begin{lem}
Suppose that the only good pair of $\beta$ is $(0,0)$ and that $\eta$ does not belong to the classes ~\ref{eqn:C} or \ref{eqn:E}(1). Then, the only good pair of $\beta\times\eta$ is $(0,0)$.
\end{lem}
\begin{proof}
Since $\eta$ is not in \ref{eqn:C} or \ref{eqn:E}(1), there must exist a vector $r \in R$ such that $\eta(r,r)\neq \eta(r,t.r)$. Then, for any scalars $a, b \in \KK$, the system of equations \begin{align*}
    a=c\eta(r,r)+d\eta(r,t.r), \\
    b=c\eta(r,t.r)+d\eta(r,r)
\end{align*} has a solution in some scalars $c$ and $d$. By Proposition ~\ref{system}, there exists a vector $u \in U$ such that $\beta(u,u)=c,\beta(u,t.u)=d$. We obtain \begin{align*}
    &(\beta\times\eta)(u\otimes r,u\otimes r)=\beta(u,u)\eta(r,r)+\beta(u,t.u)\eta(r,t.r)=c\eta(r,r)+d\eta(r,t.r)=a, \\
    &(\beta\times\eta)(u\otimes r, t.(u\otimes r)) = \beta(u, t.u)\eta(r,r)+\beta(u,u)\eta(r,t.r) = d\eta(r,r)+c\eta(r,t.r)=b.
\end{align*}
For $(k, l) \in \KK^2$ to be a good pair of $\beta\times\eta$, the equation $kb=la$ must hold for all values of $a, b$. This is only true when $(k, l)=(0, 0)$.
\end{proof}
\begin{lem}
Let $k_1, k_2, \ell_1,$ and $\ell_2$ be elements of $\KK$. Suppose that the good pairs of $\beta$ are the multiples of $(k_1,\ell_1)$ and the good pairs of $\eta$ are the multiples of $(k_2,\ell_2)$. Suppose further that $\beta$ and $\eta$ are not in ~\ref{eqn:C} or \ref{eqn:E}(1). Then, the good pairs of $\beta\times\eta$ are the multiples of $(k_1k_2+\ell_1\ell_2, k_1\ell_2+\ell_1k_2)$.
\end{lem}
\begin{proof}
First, we observe that for all $u \in U, r \in R$, 
\begin{align*}
    &(k_1k_2+\ell_1\ell_2)(\beta\times\eta)(u \otimes r, t.(u \otimes r))
    \\&\quad=(k_1k_2+\ell_1\ell_2)(\beta(u,u)\eta(r,t.r)+\beta(u,t.u)\eta(r,r)) \\
    &\quad =k_1k_2\beta(u,t.u)\eta(r, r)+k_1k_2\beta(u,u)\eta(r,t.r) + \ell_1\ell_2\beta(u, t.u)\eta(r,r)+\ell_1\ell_2\beta(u, u)\eta(r, t.r) \\
& \quad=k_2\ell_1\beta(u,u)\eta(r,r)+k_1\ell_2\beta(u,u)\eta(r,r)+k_2\ell_1\beta(u, t.u)\eta(r,t.r)+k_1\ell_2\beta(u,t.u)\eta(r,t.r)\\
    &\quad =(k_1\ell_2+\ell_1k_2)(\beta(u,u)\eta(r,r)+\beta(u,t.u)\eta(r,t.r))\\
    &\quad =(k_1\ell_2+\ell_1k_2)(\beta\times\eta)(u\otimes r,u\otimes r),
\end{align*}
which shows that the multiples of $(k_1k_2+\ell_1\ell_2,k_1\ell_2+\ell_1k_2)$ are good pairs of $\beta\times\eta$. It remains to prove that they are the only good pairs of $\beta\times\eta$. 

If $k_1\ell_2=\ell_1k_2$, then the multiples of $(1, 0)$ are good pairs of $\beta\times\eta$. If $k_1\ell_2\neq \ell_1k_2$, then the multiples of $(\frac{k_1k_2+\ell_1\ell_2}{k_1\ell_2+\ell_1k_2},1)$ are good pairs of $\beta\times\eta$. In either case, $\beta\times\eta$ will not have other good pairs unless it belongs to ~\ref{eqn:C}. We will prove that this cannot occur.

Because $\beta$ does not belong to ~\ref{eqn:C} or \ref{eqn:E}(1), there exists a vector $u' \in U$ such that at least one of $\beta(u',u'),\beta(u',t.u')$ is nonzero. Similarly, because $\eta$ does not belong to ~\ref{eqn:C} or \ref{eqn:E}(1), there exists a vector $r' \in R$ such that at least one of $\eta(r',r'),\eta(r',t.r')$ is nonzero. The quantities  $\beta(u',u')+\beta(u',t.u')$ and $\eta(r',r')+\eta(r',t.r')$ are therefore both nonzero, and their product 
\begin{align*}
    &(\beta(u',u')+\beta(u',t.u'))(\beta(r',r')+\beta(r',t.r')) \\
    &\quad =(\beta(u',u')\beta(r'r')+\beta(u',t.u')\eta(r',t.r')) + (\beta(u', t.u')\eta(r',r')+\beta(u',u')\eta(r',t.r')) \\
    &\quad =(\beta\times\eta)(u' \otimes r', u' \otimes r')+(\beta\times\eta)(u' \otimes r', t.(u' \otimes r'))
\end{align*}
must also be nonzero. At least one of $(\beta\times\eta)(u' \otimes r', u' \otimes r')$ and $(\beta\times\eta)(u' \otimes r', t.(u' \otimes r'))$ is nonzero; this cannot be the case for forms in ~\ref{eqn:C}. Hence, $\beta \hat \otimes \eta$ has no other good pairs, which proves the claim.
\end{proof}

Our work above fully determines the good pairs of $\beta\times\eta$ in the remaining cases. Now, we will find when $\beta\times\eta$ is alternating.
\begin{lem}\label{alternatingifftensor}
The form $\beta\times\eta$ is alternating if and only if at least one of $\beta$ and $\eta$ is alternating.
\end{lem}
\begin{proof}
The object $U \otimes R$ can be decomposed as $U \otimes R=(m\un \oplus nP)\otimes (p\un \oplus qP)=mp\un\oplus (2nq+mq+np)P$. A basis for $mp \un$ is given by the vectors $v_i\otimes \nu_j$ where $1\le i\le m, 1\le j\le p$. By Proposition ~\ref{alternatingfirstdef}, the form $\beta\times \eta$ is alternating when $(\beta\times\eta)(v_i\otimes \nu_j,v_i\otimes \nu_j)=0$ for all $1 \leq i \leq m, 1 \leq j \leq p$.

Expanding, we have $$(\beta\times\eta)(v_i\otimes \nu_j,v_i\otimes \nu_j)=\beta(v_i, v_i)\eta(\nu_j, \nu_j) + \beta(v_i, t.v_i)\beta(t.\nu_j, \nu_j)=\beta(v_i,v_i)\eta(\nu_j,\nu_j).$$ By Proposition ~\ref{alternatingfirstdef}, $\beta(v_i, v_i)=0$ for all $1 \leq i \leq m$ if and only if $\beta$ is alternating, and $\eta(\nu_j, \nu_j)=0$ for all $1 \leq j \leq p$ if and only if $\eta$ is alternating. Thus, $\beta\times\eta$ is alternating if and only if $\beta$ is alternating, $\eta$ is alternating, or both $\beta$ and $\eta$ are alternating.
\end{proof}

We will now describe the form invariant $\mathcal{I}_{\beta\times\eta}$ when $\beta\times\eta$ is alternating. By Propositions ~\ref{nondegenperp} and \ref{splitvfromu}, we can choose decompositions of $U$ and $R$ such that $m \un \perp nP$ and $p \un\perp qP$. This results in a decomposition $U \otimes R = mp \un \oplus mqP\oplus npP\oplus 2nqP$ where the subobjects $mp\un$, $mqP$, $npP$, and $2nqP$ are mutually orthogonal.

By Remark \ref{invariant_nP}, the form invariant of $\beta\times\eta$ is equal to the form invariant of $\beta\times\eta$ restricted to $mqP\oplus npP\oplus 2nqP$. The restrictions of $\beta\times\eta$ to $mqP$, $nqP$, and $2nqP$ are all alternating, so we can apply Lemma \ref{invariantsum} to write 
\begin{equation}\label{invariantdecomposition}    \mathcal{I}_{\beta\times\eta}=\mathcal{I}_{\beta\times\eta|_{mqP}}+\mathcal{I}_{\beta\times\eta|_{npP}}+\mathcal{I}_{\beta\times\eta|_{2nqP}}.
\end{equation} Therefore, our approach will be to determine the form invariants of the restrictions of $\beta\times\eta$ to the objects $mqP$, $npP$, and $2nqP$.
\begin{prop}\label{2nqP}
If $\beta\times\eta$ is alternating, then the form invariant of $\beta\times\eta$ restricted to $nP\otimes qP=2nqP$ is zero.
\end{prop}
\begin{proof}
 The object $2nqP$ contains the $2nq$ linearly independent vectors given by $w_i \otimes \chi_j, x_i\otimes \chi_j$ for $1 \le i\le n, 1 \le j\le q$. Observe that $t.(w_i\otimes \chi_j)=x_i\otimes \chi_j$. Now, consider $X$-function and $X$-form of $\beta\times\eta$, which we will denote by $f$ and $g$, respectively. 
For all $1\le i,k\le n,1\le j,\ell\le q$, 
 \begin{align*}
     g(x_i\otimes \chi_j,x_k\otimes \chi_\ell)&=(\beta\times\eta) (w_i\otimes \chi_j,x_k\otimes \chi_\ell)  \\&=\beta(w_i,x_k)\eta(\chi_j,\chi_\ell)+\beta(w_i,t.x_k)\eta(t.\chi_j,\chi_\ell)\\&=\beta(w_i,x_k)\eta(\chi_j,\chi_\ell)+\beta(w_i,0)\eta(0,\chi_\ell)=0.
 \end{align*} 
 
 Furthermore, for all $1\le i\le n,1\le j\le q$, $$f(x_i\otimes \chi_j)=(\beta\times\eta)(w_i\otimes \chi_j,w_i\otimes \chi_j)=\beta(w_i,w_i)\eta(\chi_j,\chi_j)+\beta(w_i,x_i)\eta(\chi_j,0)=0.$$ A basis $\{b_1,b_2,\dots ,b_{2nq}\}$ of the image of $2nqP$ under the map of the $t$-action can be constructed such that the vectors $b_{nq+1}, \dots b_{2nq}$ are given by $x_i\otimes \chi_j$, where $1 \le i\le n, 1 \le j\le q$. Using this basis, we construct the $X$-matrix of $\beta\times\eta$ restricted to $2nqP$. It is of the form
\begin{center}
$\left[\begin{tabular}{lll|lll}
&&&&&\\
&A&&&B&\\
&&&&&\\ \hline
&&&&&\\
&C&&&0&\\
&&&&&
\end{tabular}\right],$
\end{center}
where $A,B,C$ are matrix blocks and $0$ represents the zero matrix. We know by the non-degeneracy of the $X$-form (proved in Proposition ~\ref{gnondegen}) that $M$ is invertible, so $B$ and $C$ must also be invertible. We calculate that $M^{-1}$ is equal to
\begin{center}
$\left[\begin{tabular}{lll|lll}
&&&&&\\
&$0$&&&$C^{-1}$&\\
&&&&&\\ \hline
&&&&&\\
&$B^{-1}$&&&$B^{-1}AC^{-1}$&\\
&&&&&
\end{tabular}\right]$.
\end{center}
Thus, $M^{-1}_{kk}=0$ for $1\le k\le nq$ and $f(b_k)=0$ for $nq < k \le 2nq$. The form invariant of $\beta\times\eta$ restricted to $2nqP$ evaluates to $\mathcal{I}_{\beta\times\eta|_{2nqP}} = \underset{k=1}{\overset{2nq}{\sum}}f(b_k)M^{-1}_{kk}=0.$
\end{proof}

\begin{prop}\label{isnotalternating}
Suppose $\beta\times\eta$ is alternating. If $\beta$ is not alternating, then the form invariant of $\beta\times\eta$ restricted to $m \un \otimes qP=mqP$ is $m \mathcal{I}_{\eta|_{qP}}$. 
\end{prop}
\begin{proof}
    The object $m\un$ is the direct sum of $m$ $\un$ objects, for each of which the restriction of $\beta$ is non-degenerate. The object $mqP$ is the direct sum of $m$ copies of $\un \otimes qP$.  Each $\un \otimes qP$ object is alternating, so applying Lemma ~\ref{invariantsum} reduces the claim to proving that $\mathcal{I}_{\beta\times\eta|_{\un \otimes qP}}=\mathcal{I}_{\eta|_{qP}}$. This is true because $\beta\times\eta|_{\un \otimes qP} \cong \eta|_{qP}$.
\end{proof}

\begin{prop}\label{isalternating}
    Suppose $\beta\times\eta$ is alternating. If $\beta$ is alternating, the form invariant of $\beta\times\eta$ restricted to $m\un\otimes qP = mqP$ is zero. 
\end{prop}
\begin{proof}
The object $m\un$ is the direct sum of $\frac{m}{2}$ $2\un$ objects, each of which has a basis $\{u_1,u_2\}$ such that $\beta(u_1,u_1)=0$, $\beta(u_2,u_2)=0$, and $\beta(u_1,u_2)=1$. The object $2\un \otimes qP$ is alternating, and $mqP$ is the direct sum of $\frac{m}{2}$ copies of $2\un \otimes qP$. Applying Lemma ~\ref{invariantsum} to these $\frac{m}{2}$ objects, we only need to show that $\mathcal{I}_{\beta\times\eta|_{2\un\otimes qP}}=0.$ We will do so by directly calculating this form invariant.

The object $2\un\otimes qP$ contains the $2q$ linearly independent vectors given by $u_1\otimes \omega_i$ and $u_1\otimes \chi_i$, where $t.(u_1\otimes \omega_i)=u_1\otimes \chi_i$ for $1\le i\le q$. Denote the $X$-function and the $X$-form of $\beta\times\eta$ by $f$ and $g$, respectively. For $1 \le i\le q$, we have $$f(u_1\otimes \chi_i)=(\beta\times\eta)(u_1\otimes \omega_i,u_1\otimes \omega_i)=\beta(u_1,u_1)\eta(\omega_i,\omega_i)=0,$$ and for all $1 \leq i,j\le q$, we have
\begin{align*}
    g(u_1\otimes \chi_i,u_1\otimes \chi_j)=(\beta\times\eta)(u_1\otimes \omega_i,u_1\otimes \chi_j)=\beta(u_1,u_1)\eta(\omega_i,\chi_j)=0.
\end{align*}

We can construct a basis $\{b_1, b_2, \dots b_{2q}\}$ of the image of $2\un\otimes qP$ under the map of the $t$-action such that the vectors $b_{q+1}, \dots, b_{2q}$ are given by $u_1\otimes \chi_i$ for $1\le i\le q$. 
Let $M$ be the $X$-matrix of $\beta\times\eta$ on this basis. By the same reasoning used for the case in Lemma ~\ref{2nqP}, $M^{-1}_{kk}=0$ for $1\le k\le q$ and $f(b_k)=0$ for $q < k \le 2q$, which proves that $$\underset{k=1}{\overset{2q}{\sum}}f(b_k)M^{-1}_{kk}=0.$$ Having shown that the form invariant of $\beta\times\eta$ restricted to each $2\un\otimes qP$ object is zero, we also have $\mathcal{I}_{\beta\times\eta|_{m\un\otimes qP}}=0$.
\end{proof}

By commutativity, the previous two lemmas prove that $\mathcal{I}_{\beta\times\eta|_{npP}} = p\mathcal{I}_{\beta |_{nP}}$ when $\eta$ is not alternating and $\mathcal{I}_{\beta\times\eta|_{npP}}=0$ when $\eta$ is alternating. 

Given an alternating form $\beta\times\eta$, we can now find the form invariant of $\beta\times\eta$ using ~\eqref{invariantdecomposition}. At least one of $\beta$ and $\eta$ must be alternating by Lemma ~\ref{alternatingifftensor}.  By Propositions ~\ref{2nqP}, ~\ref{isnotalternating}, and ~\ref{isalternating}, $\mathcal{I}_{\beta\times \eta}=0$ when both $\beta$ and $\eta$ are alternating, $\mathcal{I}_{\beta\times \eta}=m \mathcal{I}_{\eta|_{qP}}=m\mathcal{I}_\eta$ when $\beta$ is not alternating, and $\mathcal{I}_{\beta\times \eta}=p\mathcal{I}_{\beta|_{nP}}=p \mathcal{I}_\beta$ when $\eta$ is not alternating.

Our work in this section determines the good pairs of $\beta\times\eta$, when $\beta\times\eta$ is alternating, and the form invariant of $\beta\times\eta$ when the form is alternating. This enables us to calculate the multiplication on our isomorphism classes in the table below. Again, the top row describes the isomorphism class of $\beta$ (on $m\un \oplus nP$), and the leftmost column describes the isomorphism class of $\eta$ (on $p\un \oplus qP$).

\begin{center}
\begin{tabular}{|l|l|l|l|l|l|p{50mm}|l|}
\hline
 $\times$      & A & B & C & D    & E(1) & E($a$)                                                 & F($a$)  \\ \hline
A      & A & B & C & D    & E(1) & E($a$)                                                 & F($pa$) \\ \hline
B      &   & B & C & F(0) & E(1) & F($pna$)                                                   & F($pa$) \\ \hline
C      &   &   & C & C    & C    & C                                                      & C       \\ \hline
D      &   &   &   & D    & E(1) & E($a$)                                                 &  F($0$)   \\ \hline
E(1)   &   &   &   &      & C    & E(1)                                                 & E(1)    \\ \hline
E($b$) &   &   &   &      &      & $a=b\rightarrow$ D;\newline $a\neq b\rightarrow$ E($(ab+1)/(a+b)$) &  F($0$)   \\ \hline
F($b$) &   &   &   &      &      &                                                        &  F($0$)   \\ \hline
\end{tabular}
\end{center}
In the table, we again use $a$ and $b$ to denote arbitrary scalars.

\printbibliography

\newpage
\appendix
\section{Deferred Proofs}\label{appendix}
The following is a proof of Proposition \ref{products}.
\begin{proof}
    The proof is a consequence of the hexagonal diagrams that the braiding $c$ is required to satisfy. Recall that in a symmetric tensor category, the hexagonal diagrams say for any $X, Y, Z \in \mathcal{C}$ the following two equations hold: 

    \begin{equation}\label{braid_axioms}
    \begin{aligned}
        c_{X,Y\otimes Z} = 1_Y \otimes c_{X,Z} \circ c_{X,Y} \otimes 1_Z, \\
        c_{X \otimes Y, Z} = c_{X,Z} \otimes 1_Y \circ 1_X \otimes c_{Y,Z}.
    \end{aligned}
    \end{equation}
    (see \cite[pg. 195]{etingof2016tensor}). Then, we can write 

    \begin{align*}
        (\beta \times \gamma) \circ c_{V\otimes W, V\otimes W} &= m \circ (\beta \otimes \gamma) \circ (1_{V} \otimes c_{W,V} \otimes 1_W) \circ c_{V \otimes W, V \otimes W} \\
        &= m \circ (\beta \otimes \gamma) \circ (1_{V} \otimes c_{W,V} \otimes 1_W) \circ (1_V \otimes c_{V \otimes W, W}) \circ (c_{V \otimes W, V} \otimes 1_W) \\
        &= m \circ (\beta \otimes \gamma) \circ (1_{V} \otimes c_{W,V} \otimes 1_W) \circ (1_V \otimes c_{V,W} \otimes 1_W) \\ &\phantom{=}\circ (1_{V} \otimes 1_V \otimes c_{W,W}) \circ (c_{V \otimes W, V} \otimes 1_W)
    \end{align*}
    Now, because $c_{W,V} \circ c_{V,W} = 1_{V\otimes W}$, the morphisms in the middle cancel out, and we are left with

    \begin{align*}
        (\beta \times \gamma) \circ c_{V\otimes W, V\otimes W} &= m \circ (\beta \otimes \gamma) \circ (1_{V} \otimes 1_V \otimes c_{W,W}) \circ (c_{V \otimes W, V} \otimes 1_W) \\
        &= m \circ (\beta \otimes \gamma) \circ (1_{V} \otimes 1_V \otimes c_{W,W}) \circ (c_{V,V} \otimes 1_W \otimes 1_W) \\ &\hphantom{=} \circ (1_{V} \otimes c_{W,V} \otimes 1_W) \\
        &= m \circ ((\beta\circ c_{V,V}) \otimes (\gamma \circ c_{W,W})) \circ (1_{V} \otimes c_{W,V} \otimes 1_W).
    \end{align*}
    Now, depending on whether $\beta$ and $\gamma$ are symmetric or skew-symmetric, the claim follows.
\end{proof}

The following is a proof of Proposition \ref{witt_semi_ring}.
\begin{proof}
    Most of the statement is obvious, except the commutativity, which we show now. We will show that  $[\beta \times \gamma] = [\gamma \times \beta]$ for any two bilinear forms $\beta$ and $\gamma$ on $V$ and $W$, respectively (the assumption about non-degeneracy and symmetry is not necessary). To get started, we first observe that $c_{\un, \un}: \un \otimes \un \rightarrow \un \otimes \un$ can be identified with $\lambda = \pm 1$ by Schur's lemma because we have the unit isomorphism $m :\un \otimes \un \rightarrow \un$, because $\un$ is irreducible, and because $c_{\un,\un}$ squares to the identity.
    \par
    Now, we claim that $\beta \times \gamma = (\gamma \times \beta) \circ (\phi \otimes \phi)$, where the isomorphism $\phi: V \otimes W \rightarrow W \otimes V$ is given by $\phi \coloneqq \lambda^{1/2} c_{V,W}$. To see this, we start by expanding out the definition:

    \begin{align*}
        (\gamma \times \beta) \circ (\phi \otimes \phi) &= \lambda^{-1}(\gamma \times \beta) \circ c_{V,W}\otimes c_{V,W} \\
        &= \lambda m \circ (\gamma \otimes \beta) \circ (1_W \otimes c_{V,W} \otimes 1_V) \circ c_{V,W} \otimes c_{V,W}.
    \end{align*}
    Now, by the fact that the braiding is a natural isomorphism between the tensor product functor $\otimes: \mathcal{C}\boxtimes \mathcal{C} \rightarrow \mathcal{C}$ and its opposite $\otimes^{\mathrm{op}}: \mathcal{C} \boxtimes \mathcal{C} \rightarrow \mathcal{C}$, where the opposite functor $\otimes^{\mathrm{op}}$ is given by $\otimes^{\mathrm{op}}(V \boxtimes W) = W \otimes V$ and $\boxtimes$ denotes the Deligne tensor product, we have 

    \[\gamma \otimes \beta = c_{\un, \un} \circ  (\beta \otimes \gamma) \circ c_{W \otimes W, V \otimes V}.\]
    This implies that 

    \begin{align*}\label{eq1}
        (\gamma \times \beta) \circ(\phi \otimes \phi) &= \lambda^{-1}  m \circ (c_{\un, \un} \circ  (\beta \otimes \gamma) \circ c_{W \otimes W, V \otimes V}) \circ (1_W \otimes c_{V,W} \otimes 1_V) \circ c_{V,W} \otimes c_{V,W} \\
        &= m\circ  (\beta \otimes \gamma) \circ c_{W \otimes W, V \otimes V} \circ (1_W \otimes c_{V,W} \otimes 1_V) \circ c_{V,W} \otimes c_{V,W}.
    \end{align*}
    Now, by applying both of the braid axioms in \eqref{braid_axioms} once, we know that 
    
    \[c_{W \otimes W, V \otimes  V} = (1_V \otimes c_{W,V} \otimes 1_W) \circ (c_{W,V} \otimes c_{W,V}) \circ (1_W \otimes c_{W,V} \otimes 1_V),\]
    so we deduce that 

    \begin{align*}
        (\gamma \times \beta) \circ(\phi \otimes \phi)  &=  m \circ  (\beta \otimes \gamma) \circ (1_V \otimes c_{W,V} \otimes 1_W) \\
        &= \beta \times \gamma.c
    \end{align*}
    This shows $\beta \times \gamma$ and $\gamma \times \beta$ are equivalent, so $[\beta \times \gamma] = [\gamma \times \beta]$.
 \end{proof}

 The following is a proof of Lemma \ref{triangular_structure}.
 \begin{proof}
    This is a straightforward verification of the axioms in \eqref{triangular_structure_axioms}. For instance, to see that $R$ is invertible, we notice that 
    \begin{align*}
        R^2 &= (1 \otimes 1 + t \otimes t)(1 \otimes 1 + t \otimes t)\\ 
        &= 1 \otimes 1 + 2(t \otimes t) + t^2 \otimes t^2 = 1 \otimes 1,
    \end{align*}
    so $R$ is its own inverse. We can also check that 
    \begin{align*}
        (\Delta \otimes 1_A)(R) &= (\Delta \otimes 1_A)(1 \otimes 1 + t \otimes t) \\
        &= \Delta(1) \otimes 1 + \Delta(t) \otimes t \\
        &= 1 \otimes 1 \otimes 1 + 1 \otimes  t \otimes t + t \otimes 1 \otimes t \\
        &= 1 \otimes 1 \otimes 1 + 1 \otimes  t \otimes t + t \otimes 1 \otimes t + t \otimes t \otimes t^2 \\
        &= (1 \otimes 1 \otimes 1 + t \otimes 1 \otimes t) (1 \otimes 1 \otimes 1 + 1 \otimes t \otimes t)\\
        &= R^{13}R^{23}.
    \end{align*}
    \end{proof}

\end{document}